\documentclass[11pt]{article}

\usepackage[
	persons={Nima, Vishesh, Fred, Huy, June},
	authors=footnotes,
	authorsfont=normalsize,
    bibliosources={refs.bib},
]{pomegranate}

\DeclareOperator{\diag}
\DeclareDocumentMathCommand{\corMat}{}{\Psi^{\operatorname{cor}}}

\DeclareDelimiter{\dTV}[\mathnormal{d}_{\operatorname{TV}}]{\lparen}{\rparen}
\DeclareDelimiter{\DKL}[\mathcal{D}_{\operatorname{KL}}]{\lparen}{\rparen}
\DeclareDelimiter{\Ent}[\operatorname{Ent}]{\lbrack}{\rbrack}
\DeclareDelimiter{\chis}[\chi^2]{\lparen}{\rparen}
\DeclareDelimiter{\Var}[\operatorname{Var}]{\lbrack}{\rbrack}
\DeclareDelimiter{\Ehat}[\widehat{\mathbb E}]{\lbrack}{\rbrack}

\title{Universality of Spectral Independence with Applications \\ to Fast Mixing in Spin Glasses}
\date{}

\Tag{
	\author[1]{Nima Anari}
	\author[2]{Vishesh Jain}
	\author[1]{Frederic Koehler}
	\author[1]{Huy Tuan Pham}
	\author[1]{Thuy-Duong Vuong}
	\affil[1]{Stanford University, \url{{anari,fkoehler,huypham,tdvuong}@stanford.edu}}
	\affil[2]{University of Illinois Chicago, \url{visheshj@uic.edu}}
}
\Tag<anon>{
	\author{}
}

\begin{document}
	\maketitle
	\begin{abstract}
	    We study Glauber dynamics for sampling from discrete distributions $\mu$ on the hypercube $\{\pm 1\}^n$. Recently, techniques based on spectral independence have successfully yielded optimal $O(n)$ relaxation times for a host of different distributions $\mu$. We show that spectral independence is universal: a relaxation time of $O(n)$ implies spectral independence.
	    
	    We then study a notion of tractability for $\mu$, defined in terms of smoothness of the multilinear extension of its Hamiltonian -- $\log \mu$ -- over $[-1,+1]^n$. We show that Glauber dynamics has relaxation time $O(n)$ for such $\mu$, and using the universality of spectral independence, we conclude that these distributions are also fractionally log-concave and consequently satisfy modified log-Sobolev inequalities. We sharpen our estimates and obtain approximate tensorization of entropy and the optimal $\widetilde{O}(n)$ mixing time for random Hamiltonians,
     i.e.\ the classically studied mixed $p$-spin model at sufficiently high temperature. These results have significant downstream consequences for concentration of measure, statistical testing, and learning.
      % $\log \mu$,
     %quasi-random Hamiltonians $\log \mu$. This class includes our main application, the classically studied mixed $p$-spin model in high temperatures.
	\end{abstract}
	\clearpage
 
	\section{Introduction}

In this paper, we study probability measures $\mu(\sigma)$ on the hypercube $\set{\pm 1}^n$ and the standard Markov chain for sampling from such distributions known as the Glauber dynamics. Any such distribution $\mu(x)$ with full support can be written in the form
\[ \mu(\sigma) = \frac{1}{Z} \exp\parens*{H(\sigma)} \]
for some function $H : \set{\pm 1}^n \to \R$, unique up to an additive constant, and corresponding normalizing constant $Z$. We say that such a distribution $\mu$ is the \emph{Gibbs measure for Hamiltonian} $H$. 

Directly sampling from such a distribution is not easy, because evaluating the normalizing constant $Z$ (``partition function'') can be computationally difficult. For this reason, samples are typically generated using a Markov chain approach. In particular, the \emph{Glauber dynamics} or Gibbs sampler is a natural Markov chain with stationary distribution $\mu$ that resamples coordinates one at a time:
\begin{Algorithm*}
    Let $\sigma^{(0)} \in \{\pm 1\}^n$ be the initial state.\;
	\For{$t=1,\dots$}{
		Select $i$ uniformly at random from $[n] = \set{1,\dots,n}$.\;
		Let $\sigma^{(t)}_j = \sigma^{(t - 1)}_j$ for all $j \ne i$, and sample $\sigma^{(t)}_i$ from the conditional law $\P_{\mu}{\sigma = \cdot \given \sigma_{\sim i} = \sigma_{\sim i}^{(t - 1)}}$.\;
	}
\end{Algorithm*}
Each step of this chain is easy to implement given access to the Hamiltonian $H$, and in particular does not require knowledge of the normalizing constant $Z$. What is less obvious is how long the chain should be run. Understanding the \emph{mixing time}, i.e.\ the number of steps which must be executed for the chain to approximately reach stationarity, for particular Hamiltonians $H$ is a very important and mathematically difficult task which has been intensely studied for multiple decades (see e.g.\ \cite{aldous1995reversible,martinelli1999lectures,guionnet2003lectures,levin2017markov} for background). 
It suffices to say that there are many approaches to prove rapid mixing with different strengths and weaknesses. 

\paragraph{High-dimensional expansion and spectral independence.}Recently, a fruitful approach to analyzing the Glauber dynamics has emerged based on connections to high-dimensional expansion and the geometry of polynomials. %In this approach the Glauber dynamics is viewed as an instance of a ``down-up walk'' on a pure simplicial complex. 
These ideas have been used to establish optimal mixing time bounds in many settings (see e.g.\ \cite{anari2019log,cryan2019modified,ALO20,blanca2021mixing,chen2020rapid,chen2021optimal,anari2021entropic,anari2021entropic2,chen2021rapid,alimohammadi2021fractionally,abdolazimi2022matrix,chen2022spectral,chen2022optimal}). 

One of the key concepts in this approach is \emph{spectral independence}. 
A Gibbs measure $\mu$ is $\eta$-\emph{spectrally independent} if $\lambda_{\max}(\Psi) \le \eta$ where $\Psi \in \mathbb{R} ^{2n \times 2n}$ is the correlation matrix indexed by pairs in $[n] \times \{\pm 1\}$ with entries
\[ \Psi_{(i,\tau_i),(j,\tau_j)} = \P_{\mu}{\sigma_i = \tau_i \mid \sigma_j = \tau_j} - \P_{\mu}{\sigma_i = \tau_i}. \]
Proving that the Gibbs measure (and all of its conditionings) is $O(1)$-spectrally independent is the key first step for applying the high-dimensional expansion approach to mixing time (e.g.\ to apply Theorem 1.3 of \cite{ALO20}, which implies polynomial mixing time bounds). 

Understanding this concept and finding ways to prove spectral independence is therefore very useful. Many authors have asked and studied how spectral independence relates to existing notions in the sampling literature. For example, Dobrushin's uniqueness criterion \cite{dobrushin1968problem} is a classical concept which considers the properties of a related-looking influence matrix. Recently, \textcite{liu2021coupling} and \textcite{blanca2021mixing} proved that Dobrushin's criterion (and more generally, the existence of a contractive coupling) implies $O(1)$-spectral independence. A similar implication holds under other widely used criteria such as correlation decay \cite{chen2020rapid} or zero-freeness \cite{alimohammadi2021fractionally, chen2022spectral}.

%where up to an additive constant $H(x)$ is the log-density $\log \mu(x)$.   
\subsection{Our Results}
\paragraph{Universality of spectral independence.} 
We prove a new result which directly connects spectral independence with the classical notion of spectral gap. 
The \emph{spectral gap} or Poincar\'e constant of a Markov chain is the gap between the two largest eigenvalues of its transition matrix $P$. In other words, we say that $P$ has \emph{spectral gap} $\lambda$ if $\lambda_2(P) = 1 - \lambda$. This is a key concept in the study of Markov chains --- informally, the spectral gap controls the speed of mixing from a warm start. For this reason, the inverse spectral gap $1/\lambda$ is known as the \emph{relaxation time} of the (lazy version of the) Markov chain \cite{levin2017markov}.

%We prove that spectral independence is implied by a large spectral gap of the Glauber dynamics.
We prove that \emph{spectral independence is a relaxation of the spectral gap}, in other words that $O(1)$-spectral independence necessarily holds if the Glauber dynamics has a large spectral gap.
\begin{theorem}
%Let $P$ be the transition matrix for a single step of the Glauber dynamics for probability measure $\mu$ on $\{\pm 1\}^n$. 
If the Glauber dynamics of $\mu$ has spectral gap $\frac{1}{C n}$,
%, i.e.\ $\lambda_2(P) \le 1 - \frac{1}{C n}$, 
then $\mu$ is $C$-spectrally independent. 
\end{theorem}
We actually prove a stronger fact (\cref{thm:poincare-to-si}), which is the natural generalization of this result to down-up walks. In particular, this means that the same relation holds for the Glauber dynamics with spins valued in arbitrary alphabets, not just binary spins. 

Conceptually, this gives a simple explanation for the ubiquity of spectral independence: it is a \emph{necessary} condition for the Glauber dynamics to have $O(n)$ relaxation time. It immediately implies that spectral independence holds in a large number of settings where mixing time analysis was performed using other methods (e.g., via stochastic localization \cite{eldan2021spectral,anari2021entropic,chen2022localization} or curvature arguments \cite{caputo2015approximate,erbar2017ricci}). This in turn has further nontrivial consequences: if the relaxation time is $O(n)$ under all external fields, we get fractional log-concavity of the generating polynomial, which in turn implies subadditivity of entropy and Brascamp-Lieb type inequalities  \cite[see][]{anari2021entropic,alimohammadi2021fractionally,barthe2011correlation,blanca2021mixing}. 

\paragraph{Rapid mixing from smoothness of the multilinear extension.} 
Building on the universality of spectral independence, we are able to prove new results concerning the mixing of the Glauber dynamics. We would like to understand what conditions on $H$ naturally lead to rapid mixing of the Glauber dynamics. For inspiration, we know that for \emph{continuous distributions} on $\mathbb{R}^n$ and the corresponding (continuous time) \emph{Langevin dynamics}, strong log-concavity of the distribution implies rapid mixing via Bakry-Emery theory \cite{bakry2014analysis}. For a distribution with smooth density $p$, this just means that $\nabla^2 \log p \preceq - \epsilon I$ for some $\epsilon > 0$. 

We identify a natural analogue of this fact on the discrete hypercube. First (as in e.g.\ \cite{eldan2018decomposition}), we identify $H$ with its \emph{multilinear extension} $H : \mathbb{R}^n \to \mathbb{R}$ defined by 
\[ H(x) = \sum_{S \subset [n]} \hat H(S) \prod_{i \in S} x_i \]
where $\hat H(S) := \frac{1}{2^n} \sum_{\sigma \in \{\pm 1\}^n} H(\sigma) \prod_{i \in S} \sigma_i$ is the Fourier transform of $H$ viewed as a function on the hypercube \cite{o2014analysis}.
We prove that as long as $\nabla^2 H$ is spectrally small (i.e.\ $H$ is sufficiently smooth in the usual sense), Glauber dynamics mixes rapidly:
\begin{theorem}[Combined \cref{thm:poincare bound} and \cref{thm:at-general}]\label{thm:general-intro}
There exist absolute constants $A,B > 0$ for which the following holds. 
Suppose that $\mu$ is a probability measure with full support on the hypercube $\{\pm 1\}^n$, so $\mu(x) \propto \exp(H(x))$ for some function $H : \{\pm 1\}^n \to \mathbb{R}$. Suppose furthermore that 
\[ \beta := \max_{\sigma \in \{\pm 1\}^n} \|\nabla^2 H(\sigma)\|_{\operatorname{op}} \le A. \]
Then we have that:
\begin{enumerate}
    \item The spectral gap of the Glauber dynamics on $\mu$ is at least $\frac{1}{(1 + B\beta)n}$.
    \item $\mu^{\hom}$ is $\frac{1}{1 + B\beta}$-fractionally log-concave.
    \item $\mu$ satisfies approximate tensorization of entropy with constant $C = O(n^{B\beta})$.
    \item The Glauber dynamics on $\mu$ satisfies the Modified Log-Sobolev Inequality (MLSI) with constant $\rho = \Omega(1/n^{1 + B\beta})$. 
    \item The Glauber dynamics on $\mu$ satisfies 
    \[ \tau_{\operatorname{mix}}(\epsilon) = O\left(n^{1 + B\beta}[\log\log(1/\min_{\sigma} \mu(\sigma)) + \log(1/\epsilon)]\right). \]
\end{enumerate}
\end{theorem}
We formally define the concepts of approximate tensorization, fractional log-concavity, etc.\ appearing here in the preliminaries (\cref{sec:preliminaries}). 
Note that the first conclusion by itself only implies $\widetilde{O}(n^2)$ mixing time, whereas the final conclusion gives a much better bound (the constant $A$ under which we can prove the spectral gap inequality is relatively small, and in particular $AB < 1$). 
\begin{remark}[Tightness]
It is not true that $\mu^{\hom}$ is a log-concave distribution (in the sense of \cite{anari2019log}) under the assumptions of this theorem --- the relaxation to fractional log-concavity is required.  Relatedly, the functional dependence of the Poincar\'e constant on $\beta$ in conclusion (1) is optimal. By this we mean that $\beta$ cannot be replaced by any function which is $o(\beta)$ as $\beta \to 0$, e.g.\ by $\beta^2$. This can be seen by examining one of various examples where the spectral gap is known exactly --- in particular the Ising model on a cycle (Theorem 15.5 of \cite{levin2017markov}). Going back to the first point, if $\mu^{\hom}$ were log-concave then this would imply  a spectral gap of at least $1/n$ \cite{anari2019log}, which is not true. Finally, the result cannot be true for any value of $A > 1$, because then it would include the supercritical Curie-Weiss model for which mixing takes exponential time \cite{levin2017markov}.
%, showing the relaxation to fractional log-concavity is required.
\end{remark}
\begin{remark}[Need for a two-sided assumption]
Comparing to the continuous setting, we might guess that this result would be true under only a one-sided bound on $\nabla^2 H$, allowing for arbitrary large negative eigenvalues (somewhat in the spirit of Equation (4) of \cite{eldan2022log}). For example, if we define a Gibbs measure with respect to the standard Gaussian distribution with Radon-Nikodym derivative proportional to $\exp(H(x))$, then we only need $\nabla^2 H \prec (1 - \epsilon) I$ to apply the Bakry-Emery criterion \cite{bakry2014analysis}.
However, we know from \cite{koehler2022sampling} that already in the case of quadratic/Ising interactions, large negative eigenvalues of $H$ can make the spectral gap large and even make sampling computationally hard.
%in the quadratic/Ising case where the Hessian is constant.  
\end{remark}

\begin{remark}
Bounds on $\beta$ can be directly obtained from the tensor injective norm of the coefficients of $H$. We can apply existing results to control the coefficient tensor for natural random examples, like the mixed $p$-spin model discussed below, or when $H(x)$ counts the number of satisfied constraints in random $k$-XOR sat (in which case, the tensor will generally be sparse), see e.g.\ \cite{zhou2021sparse}.
\end{remark}
%\Fred{somewhere briefly mention the results apply to natural sparse models, like random xor-sat. Can cite \cite{tensornorm}}
%Here $\nabla^2 H$ denotes the usual Hessian operator applied to the \emph{multilinear extension} of $H$. See Section~\ref{sec:gibbs}.

\iffalse
\begin{enumerate}
    \item We show that for any $H$ satisfying..., the Gibbs measure $\mu$ is fractionally log-concave and satisfies the dimension-free Poincar\'e inequality. (Check: the condition is basically $\|\nabla^2 H\|_{OP}$ is small? Later discuss parallel to continuous case, connections with different notions of log-concavity)
    \item We show that a subclass of quasirandom $H$, including the classically studied mixed $p$-spin model at sufficiently high temperature, satisfy the dimension-free Modified Log-Sobolev Inequality. In fact, we prove a stronger property known as \emph{approximate tensorization of entropy}. 
\end{enumerate}
\fi
%Movd later
%We build on a long line of previous work. Most recently \cite{pspin}  established the Poincar\'e inequality for the mixed $p$-spin model at high temperature. We also build on ideas developed in a long line of work studying the mixing time of Markov chains from the perspective of generating polynomials and high-dimensional expanders \cite{?}.

\paragraph{Optimal Mixing in the Mixed $p$-Spin Model.} One of the fundamental models in statistical physics corresponds to the case where the Hamiltonian $H$ has random coefficients. In other words,
\[ H(\sigma) = \sum_{p = 2}^{\infty} \frac{\beta_p}{n^{(p - 1)/2}} \sum_{1 \le i_1, \ldots, \le i_p \le n} g_{i_1 \cdots i_p} \sigma_{i_1} \cdots \sigma_{i_p} + \sum_i h_i \sigma_i \]
where the sum ranges over distinct $i_1,\ldots,i_p$, each of the coefficients $g_{i_1 \cdots i_p}$ is i.i.d.\ standard Gaussian, and we allow the external fields $h_i$ to be arbitrary (in this work, they are allowed to depend on $g$). 
The corresponding (random) measure $\mu \propto \exp(H)$ on $\{\pm 1\}^{n}$ is known as the \emph{mixed $p$-spin model} and it has been deeply studied in spin glass theory (see e.g.\ \cite{talagrand2010mean,panchenko2013sherrington} for rigorous results). 

In recent work (see below for more discussion of related work), \textcite{pspin} established an $\widetilde{O}(n^2)$ mixing time for this model at sufficiently high temperature (small $\beta$). \Cref{thm:general-intro} implies an improved mixing time bound of $\widetilde{O}(n^{1 + O(\beta)})$ for this model, where $\beta := \sum_{p\geq 2}\sqrt{p^3 \log{p}}\cdot \beta_p$. Our next result improves this to the optimal mixing time bound by proving that approximate tensorization of entropy (and hence, a modified Log-Sobolev inequality) holds with a dimension-free constant.
\begin{theorem}
\label{thm:mix p spin main}
There exists an absolute constant $A > 0$ for which the following holds. Suppose $\beta_0 := \sum_{p \geq 2}\sqrt{p^3 \log p}\cdot \beta_p \leq A.$ Let $\beta := \sum_{p \geq 2}\sqrt{2^p p^3 \log p}\cdot \beta_p$. Then, for the mixed $p$-spin model $\mu$,
with probability $\geq 1 -\exp(-\Theta(N))$ over the randomness of the Gaussian interaction terms $g$,
\begin{enumerate}
    \item $\mu$ (equivalently $\mu^{\hom}$) satisfies approximate tensorization of entropy with constant $C = O_\beta(1)$.
    \item The discrete-time Glauber dynamics on $\mu$ satisfies the Modified Log-Sobolev Inequality (MLSI) with constant $\rho = \Omega_\beta(1/n)$. 
    \item The discrete-time Glauber dynamics on $\mu$ satisfies
    \[ \tau_{\operatorname{mix}}(\epsilon) = O_\beta\parens*{n(\log\log(1/\min_{\sigma} \mu(\sigma)) + \log(1/\epsilon)}. \]
\end{enumerate}
\end{theorem}

\paragraph{Downstream Applications.} The results we establish have a number of interesting downstream applications. We discuss a few in particular:
\begin{itemize}
    \item \textbf{Consequences of the MLSI.} We obtain bounds on the Modified Log-Sobolev constant which have a number of useful consequences besides mixing time analysis. First of all, the MLSI gives more precise control on the dynamics of mixing in the form of \emph{reverse hypercontractivity} --- see \cite{mossel2013reverse,gross2014hypercontractivity}. Also, the MLSI implies subgaussian concentration of Lipschitz functions and transport-entropy inequalities \cite{van2014probability,gross2014hypercontractivity}. For example, we have established for the first time Lipschitz subguassian concentration for the high temperature mixed $p$-spin model.
    \item \textbf{Spectral independence from coupling.} Existence of a contractive coupling was previously known to imply spectral independence \cite{liu2021coupling,blanca2021mixing}. Since contractive couplings imply a spectral gap, we recover this result from the universality of spectral independence (\cref{thm:poincare-to-si}). The resulting proof is much shorter and has some notable quantitative advantages --- we discuss this more in \cref{s:simplified}. 
    \item \textbf{Learning graphical models.} It was recently observed that there is a useful connection between approximate tensorization of entropy and classical algorithms for learning distributions from data \cite{koehler2022statistical}. Combining this connection with our results, we are able to dramatically improve the state of the art results for learning some types of graphical models from samples (\cref{thm:learning ising}). For example, we are able to learn the high-temperature SK model in total variation distance from  $n^{3 + O(\beta)}$  samples in polynomial time, which is close to the best information-theoretic guarantee known \cite{devroye2020minimax}. In comparison, the best previous algorithmic results \cite{klivans2017learning,vuffray2016interaction,wu2019sparse} could guarantee sample and time complexity of only $e^{O(\beta \sqrt{n})}$.
    \item \textbf{Identity testing.} \textcite{BCSV21} recently proved implications of approximate tensorization of entropy for identity testing of high-dimensional distributions in the \emph{coordinate oracle} and \emph{subcube oracle} query models. Combining our new results with their framework gives improved sample complexities for solving many testing problems (e.g.\ \cref{cor:SK model identity testing}) -- see \cref{sec:identity} for more details. 
\end{itemize}
\subsection{Techniques}
\paragraph{Universality of spectral independence.} Interestingly, the proof of this result is very short once the correct definitions are in place, so we refer the reader to \cref{sec:general-facts}. 

\paragraph{Rapid mixing for smooth Hamiltonians.} The proof of this result combines three key ingredients: (1) universality of spectral independence, (2) a recursive spectral gap estimate of \textcite{pspin}, and (3) a large body of existing tools from the high-dimensional expanders framework. 

First, we want to prove that under the smoothness condition, the Glauber dynamics has an $\Omega(1/n)$ spectral gap. In the work \cite{pspin}, the authors developed a recursive method to prove spectral gap bounds for high temperature systems, \emph{provided} that the spectral gap of systems with $n$ spins is not much worse than that with $n - 1$ spins. To make their argument work, they needed to prove a corresponding ``continuity estimate'', but their argument used facts specific to the mixed $p$-spin model (random $H$). Our key insight is that combining the universality of spectral independence with Oppenheim's trickle-down theorem \cite{Opp18} and the local-to-global argument \cite{KO18,AL20} proves that such continuity estimates \emph{automatically} hold in a very general setting (\cref{thm:continuity}). This lets us prove the spectral gap inequality for all models satisfying the smoothness condition.

By itself, the spectral gap bound only implies $\widetilde{O}(n^2)$ time mixing. To improve the mixing time bound we appeal to the universality of spectral independence again, which lets us derive fractional log-concavity of the generating polynomial \cite{alimohammadi2021fractionally}. This implies entropic independence which in turn implies the approximate tensorization of entropy estimate via the result of \cite{anari2021entropic}. The MLSI and mixing time bound are immediate consequences (see \cref{sec:preliminaries}).

%\paragraph{Obstruction  to basic local to global for Gaussian interaction matrix}
\paragraph{Improved bound for random interactions (mixed $p$-spin model).}
First, we discuss the special case of the Sherrington-Kirkpatrick (SK) model (pure $2$-spin model).
%In order to shave off the factor $N^{\epsilon},$ we want to say that when pinning $N-k$ vertices, the spectral independence of the system becomes smaller as $k$ becomes smaller. However, this is not necessarily true.
For sufficiently high temperature i.e.\ $1/\beta \gg 1/\epsilon,$ the interaction matrix has operator norm at most $1+\epsilon$, thus the discussion above allows us to directly lower bound the modified log Sobolev constant by $ 1/n^{1+\epsilon}.$ However, this is still a factor of $n^{\epsilon}$ away from the conjectured modified log Sobolev constant of $1/n,$ when the interaction terms are i.i.d.\ Gaussians. How can we close the gap?
% in the Sherrington-Kirkpatrick model (pure 2-spin model).

%For $2$-spin system,
In the case of the SK model, 
after pinning $N-k$ spins, the resulting Hamiltonian $H^{(k)}$ is a $k\times k$ matrix with i.i.d.\ Gaussian entries of variance $1/N$, thus the corresponding operator norm scales approximately like $ \sqrt{k/N},$ so that the local to global argument gives
\[\text{MLSI constant} \geq \exp\parens*{-\sum_{k}\frac{(1+\|H^{(k)}\|_{\operatorname{op}})}{k}} \approx \exp\parens*{-\log n + \sum_{k} \frac{1}{\sqrt{kN}}} = n^{-1} \exp(-O(1)),\]
which is the desired lower bound for the modified log-Sobolev constant.

Unfortunately, when higher degree terms are present, pinning might still result in a subsystem with essentially the same operator norm. %for $p$-spin system
To see this, consider a pure $3$-spin model, i.e.\ the Hamiltonian only contains terms of degree $3$:
\[H(\sigma) = \frac{\beta}{N}\sum_{ijk} g_{ijk} \sigma_i \sigma_j \sigma_k,\]
where $g_{ijk}$ are standard i.i.d.\ Gaussians and $\beta > 0$ is a small constant. It is well known that $\norm{H} = \Theta(1)$ with high probability. Now, for any $\sigma \in \set{\pm 1}^{n}$, observe that pinning $\sigma_{[n]\setminus \set{1,2}}$, % to $\sigma,$
results in a subsystem on two spins with interaction matrix $H^{(2)}$ given by $H^{(2)}_{1,2} = \frac{1}{N}\sum_{k\not\in \set{1,2}} g_{12k} \sigma_k$, and note that with the choice of pinning $\sigma_k := \text{sign}(g_{12k})$,  
\[H^{(2)}_{1,2} = \frac{1}{N}\sum_{k\not\in \set{1,2}} \abs{g_{12k}} = \Theta(1) \text{ w.h.p}.\]
%On the other hand, for a uniformly random $\sigma \sim \set{\pm 1}^{N-2},$ $H^{(2)}_{1,2}$ is a Gaussian with variance $1/\sqrt{N},$ so \emph{on average}, we should expect a better bound on the operator norm of subsystem where only few spins are unpinned. 

To circumvent the existence of these ``bad pinnings'', we switch to an ``average-case'' version of the local to global argument \cite{alimohammadi2021fractionally,avg-local-global}. Roughly speaking, for each $i\in [N]$ and $\sigma \in \set{\pm 1}^{N\setminus \{i\}},$ we want to establish that when pinning a \emph{random} subset of $N-k$ spins according to $\sigma,$ the operator norm of the conditional subsystem is sufficiently small, e.g.\ it scales like $\sqrt{k/N}$ with high probability. We emphasize that while the average case argument allows us to consider subsystems resulting from pinning a random subset of spins according to $\sigma$, we must still consider \emph{all} possible $\sigma$; in particular, if we are to use the union bound over $\sigma$, we need to avoid bad events with exponentially (in $n$) small probability, even though the quantities involved (norms of subsystems on $k$ spins) have order governed by $k$. Moreover,  unlike in \cite{avg-local-global}, there are choices of $i$ and $\sigma$ for which we cannot expect the ``good event'' of avoiding a bad link to hold with a very high probability, e.g. $1-1/N^{10}$. Indeed, going back to the pure $3$-spin example above, for the choice of $\sigma$ there and for $i=1$, $H^{(2)}_{1,2} = \Theta(1)$, so that with probability $\Omega(k/N)$, a random pinning of $N-k$ spins includes vertex $2$ among the free spins and thus results in a subsystem with operator norm $\Theta(1)$ (instead of $O(\sqrt{k/N})$, as we might hope for). To summarize, we have to address two challenges: (i) prove a ``norm-decay'' statement for random $k\times k$ sub-matrices which holds with probability exponentially small in $n$ (as opposed to $k$), and (ii) overcome the very heavy-tailed nature of the norm of random $k\times k$ sub-matrices for use in the average case local-to-global argument \cite{alimohammadi2021fractionally, avg-local-global}. We note that while the expected norm-decay of random submatrices of a matrix has been extensively studied in the random matrix theory literature (see, e.g.~\cite{rudelson2007sampling}), our requirement that the probability bounds hold with exponentially small probability in $n$ (as opposed to $k$) makes existing techniques ineffective and needs a completely different argument.    

We conclude by briefly discussing the ideas needed to overcome these challenges. As later detailed in \cref{sec:gibbs}, the average case local-to-global argument requires us to show that for any $\sigma$, in a random ordering of pinnings according to $\sigma$, a quantity roughly of the form $\E{\exp(\sum_k \|\nabla^2 H_k\|_{\operatorname{op}}/k)}$ is bounded above by an absolute constant; here, $H_k$ is the induced Hamiltonian on the $k$ unpinned spins and the randomness is induced by the random ordering. Using standard concentration estimates for the norm of random sub-Gaussian matrices, we can bound the contribution of terms of degree at least $3$ in the induced Hamiltonian by $(k/N)^{\Omega(1)}$, so that the main challenge lies in bounding the contribution of quadratic term in the induced Hamiltonian. As explained above, this term can have operator norm $1$ with very large probability (at least $(k/N)$) and hence, we need a significantly more careful analysis of the average-case local to global iteration than in \cite{avg-local-global}. In particular, we show that except with exponentially small (in $n$) probability over the randomness of the Hamiltonian, for all $\sigma$, over the randomness of the random ordering of pinnings, $\|\nabla^2 H_k\|_{\operatorname{op}}$ is bounded by a small constant $c$ with probability $1$ (this is straightforward, given existing results), and crucially, it is bounded by $(k/N)^{1/2 - \alpha}$ with probability at least $1-(k/N)^{\alpha}$, for some $\alpha \in [c,1/2)$. Even then, we have to control $\exp(\sum_k \|\nabla^2 H_k\|_{\operatorname{op}}/k)$, where the challenge is in the complex dependencies among $H_k$ (arising from one random permutation of the pinning order). We achieve this by breaking terms based on the typical and tail behavior, and combine these heavy-tailed random variables in a careful way with an appropriately weighted H\"older's inequality. 

%We overcome the first challenge using a combination of tools from high-dimensional probability and non-asymptotic random matrix theory, such as Seginer's theorem on norms of random matrices \cite{randnorm} as well as various symmetrization and truncation arguments, and the second challenge using a significantly more careful analysis of the average-case local to global iteration than in \cite{avg-local-global}.

\subsection{Related Work}
Many works have studied techniques for analyzing discrete Markov chains which were inspired
by Bakry-Emery theory and related notions, see e.g.\ \cite{erbar2017ricci,ollivier2009ricci,eldan2017transport,caputo2009convex}
and references within for a few examples. These are quite different in nature from our results,
and in particular they have not been used to obtain results for systems with spin glass interactions like the Sherrington-Kirkpatrick and mixed $p$-spin models from statistical physics (see e.g.\ \cite{talagrand2010mean} for more background on these models). 

There have been many recent works studying sampling problems in spin glass models. For spherical spin glasses, \textcite{gheissari2019spectral} showed that the Bakry-Emery criterion can be applied to prove mixing of the Langevin dynamics at high temperature. For the Sherrington-Kirkpatrick (SK) model on the hypercube, after a line of recent works \cite{bauerschmidt2019very,eldan2021spectral,anari2021entropic} the optimal $O(n \log n)$ time mixing bound is known up to inverse temperature $\beta < 0.25$ (see also \cite{chen2022localization} for an alternative proof). All of these works only need that the interaction matrix has sufficiently small operator norm, which is exactly the same as the assumption of \cref{thm:general-intro} specialized to the case of Ising models. 
Interestingly, although the proof in \cite{anari2021entropic} is partially based on the high-dimensional expansion approach, none of these works could prove that approximate tensorization of entropy holds (with a dimension-free constant). The present work finally proves that approximate tensorization of entropy holds for sufficiently small $\beta$.  

Using a different algorithmic approach, the works \cite{alaoui2022sampling,celentano2022sudakov} constructed a sampler which sample up to the conjectured sharp threshold $\beta < 1$ in the SK model, albeit in a weaker metric (sublinear Wasserstein). Finally, as discussed above, the recent work of \textcite{pspin} which we build upon  proved spectral gap for the mixed $p$-spin model for sufficiently small $\beta$ (i.e.\ at sufficiently high temperature).  

The works \cite{eldan2022log,alimohammadi2021fractionally} and references within studied a few related analogues of (semi, fractional, etc.) ``log-concavity'' on the discrete hypercube. In this work, we have explicitly showed one such notion (fractional log-concavity of the generating polynomial) follows from a condition on the Hessian of the log-likelihood in the spirit of strong log-concavity. 

	\section{Preliminaries}
\label{sec:preliminaries}
In this paper, we use the perspective on sampling arising from the theory of generating polynomials and high-dimensional expansion. This means, for example, that the Glauber dynamics on $\{\pm 1\}^n$ is interpreted via homogenization as the $n \leftrightarrow n - 1$ down-up walk on $\binom{[2n]}{n}$ (see \cref{s:functional-inequalities} below). We discuss this and other important background in this preliminaries section.
\subsection{Down-Up Walk, Links, and Trickle-Down}
\begin{definition}[Down operator]\label{def:down-operator}
	For a ground set $\Omega$, and  $\card{\Omega} \geq k\geq \l$,  the down operator $D_{k\to \l}\in \R^{\binom{\Omega}{k}\times \binom{\Omega}{\l}}$ is defined to be
	\[ 
		D_{k\to \l}(S, T)=\begin{cases}
			\frac{1}{\binom{k}{\l}}&\text{ if }T\subseteq S,\\
			0&\text{ otherwise}.\\
		\end{cases}
	\]
\end{definition}
Note that $D_{k\to \l}D_{\l\to m}=D_{k\to m}$. 

%Note that the down operator has no dependence on $\mu$. In contrast the up operator as defined below depends on $\mu$ and is actually designed to be the time-reversal of the down operator w.r.t.\ the background measure $\mu$.
\begin{definition}[Up operator]
	For a ground set $\Omega$, $\card{\Omega}\geq k\geq \l$, and density $\mu:\binom{\Omega}{k}\to \R_{\geq 0}$, the up operator $U_{\l \to k}\in \R^{\binom{\Omega}{\l}\times \binom{\Omega}{k}}$ is defined to be
	\[ 
		U_{\l\to k}(T, S)=\begin{cases}
			\frac{\mu(S)}{\sum_{S'\supseteq T}\mu(S')}&\text{ if }T\subseteq S,\\
			0&\text{ otherwise}.\\
		\end{cases}
	\]
\end{definition}

\begin{definition}[Down-up walk]
	For a ground set $\Omega$, $|\Omega|\geq k\geq \l$, and density $\mu:\binom{\Omega}{k}\to \R_{\geq 0}$, the $k\leftrightarrow \l$ down-up walk is defined by the row-stochastic matrix $D_{k\to \l} U_{\l \to k}$. Similarly, the up-down walk is defined by $U_{\ell \to k} D_{k \to \ell}$.
\end{definition}

\begin{proposition}[{\cite[see, e.g.,][]{KO18,AL20,ALO20}}]
\label{prop:reversibility-down-up}
	The operators $D_{k\to \l}U_{\l\to k}$ and $U_{\l\to k}D_{k\to \l}$ both define Markov chains that are time-reversible and have nonnegative eigenvalues. 
	Moreover $\mu$ and $\mu D_{k \to \ell}$ are respectively their stationary distributions.
\end{proposition}

\begin{definition}[Link]
    Given $\mu$ a distribution on $\binom{\Omega}{k}$, and a set $T \subset \Omega$ with $\card{T} \le k$, we define $\mu_T$, the \emph{induced distribution on the link of $T$}, to be the probability distribution over ${\Omega \setminus \binom{T}{k-\card{T}}}$ given by the conditional law $\mu_T(S') = \P_{S \sim \mu}{S = S' \cup T \mid T \subset S}$.
\end{definition}

We will slightly abuse terminology and often refer to the induced distribution of $\mu$ on the link of $T$ simply as the link of $\mu$ at $T$. 

\paragraph{Oppenheim's Trickle-Down.} Oppenheim's trickle-down theorem inductively bounds the high-dimensional expansion of simplicial complexes, i.e.\ the spectral gap of certain up-down walks.
\begin{theorem}[\cite{Opp18}]\label{thm:oppenheim}
Let $\mu$ be a distribution on $\binom{\Omega}{k}$ with $k \ge 3$, let $P$ denote the corresponding $1 \leftrightarrow k$ up-down walk, and suppose that $\lambda_2(P) < 1$.
For $i \in \Omega$, let $\mu_i$ denote the link of $\mu$ at $i$,
and let $P_i$ denote the corresponding $1 \leftrightarrow k - 1$ up-down walk.
Suppose that for all $i \in \Omega$, $\lambda_2(P_i) \le \lambda$. Then
\[ \lambda_2(P) \le \frac{(1 - 2/k)\lambda}{1 - \lambda}.\]
\end{theorem}
\begin{remark}[Lazy vs active walk]
With our definition, the $1 \leftrightarrow k$ up-down walk has a probability of $1/k$ of staying at the same vertex. If $P$ denotes this ``lazy'' walk and $P'$ denotes the ``active'' walk which always moves to a new vertex, we have $P = \frac{1}{k} I + \frac{k - 1}{k} P'$ so $\lambda_2(P) = \frac{1}{k} + \frac{k - 1}{k} \lambda_2(P')$.
%, equivalently $\lambda_2(P') = \frac{k\lambda_2(P) - 1}{k - 1}$.
The more common ``active'' form of trickle-down states that if $\lambda_2(P') < 1$ and $\lambda_2(P'_i) \le \lambda'$ for all $i$, then $\lambda_2(P') \le \frac{\lambda'}{1 - \lambda'}$. Noting that $\lambda_2(P) < 1$ if and only if $\lambda_2(P') < 1$ and $\lambda_2(P_i) \leq \lambda$ if and only if $\lambda_2(P_i') \leq \lambda' := (\lambda - (1/k-1))\cdot (k-1)/(k-2)$, we have
\begin{align*} \lambda_2(P) 
\le \frac{1}{k} + \frac{k - 1}{k} \frac{\lambda'}{1 - \lambda'}
&= \frac{1}{k} + \frac{k - 1}{k} \frac{\frac{k - 2}{k - 1} \lambda'}{\frac{k - 2}{k - 1} - \frac{k - 2}{k - 1} \lambda'} \\
&= \frac{1}{k} + \frac{k - 1}{k} \frac{\lambda - 1/(k - 1)}{1 - \lambda}
= \frac{1}{k} + \frac{(k - 1) \lambda - 1}{k - k\lambda}
=  \frac{(1 -2/k) \lambda}{1 - \lambda},
\end{align*}
matching our statement of the trickle-down theorem. 
%.
\end{remark}

\subsection{Generating Polynomial, Tilts, and Fractional Log-Concavity}

\begin{definition}
The multivariate generating polynomial $g_{\mu} \in \R[z_1,\dots, z_n]$ associated to a density $\mu\colon 2^{[n]}\to \R_{\geq 0}$ is given by
\[g_{\mu}(z_1,\dots, z_n) := \sum_{S}\mu(S)\prod_{i\in S}z_i = \sum_{S}\mu(S)z^{S},\]
\end{definition}
Here we have used the standard notation that for $S \subseteq [n]$, $z^S = \prod_{i\in S}z_i$.

\begin{definition}[Measure tilted by external field]\label{def:measure-tilt}
For a distribution $\mu$ on $2^{[n]}$ and vector $\lambda = (\lambda_1,\dots, \lambda_n) \in \R^{n}_{>0}$, which we refer to as the \emph{external field}, we denote the measure $\mu$ tilted by external field $\lambda$ by the notation $\lambda \ast \mu$, formally defined as
%$\lambda$ external field applied to $\mu$ is a distribution on $2^{[n]}$, denoted by $\lambda \ast \mu$, given by
\[\P_{\lambda \ast \mu}{S} = \frac{1}{Z_{\lambda}} \mu(S)\cdot \prod_{i \in S}\lambda_i,\]
where the normalizing constant $Z_{\lambda}$ is defined so that $\lambda \ast \mu$ is a probability measure.
Note that for any $(z_1,\dots, z_n) \in \R^{n}_{\geq 0}$, 
\[g_{\lambda \ast \mu}(z_1,\dots, z_n) \propto g_{\mu}(\lambda_1 z_1,\dots, \lambda_n z_n).\]
%We also use the following shorthand: for $\lambda \in \R_{> 0}$, the notation $\lambda \ast \mu: = (\lambda, \ldots, \lambda) \ast \mu$ denotes the measure $\mu$ tilted by uniform external field $\lambda$. % not used in this paper
\end{definition}

In \cite{alimohammadi2021fractionally}, the notion of fractional log-concavity of the multivariate generating polynomial was developed, generalizing the concept of log-concave polynomials (see e.g.\ \cite{anari2019log}). 

\begin{definition}[Fractional log-concavity]
Consider a homogeneous distribution $\mu: \binom{[n]}{k} \to \R_{\geq 0}$ and let  $g_{\mu}(z_1, \dots, z_n)$ be its multivariate generating polynomial. For $\alpha \in [0,1]$, we say that $\mu$ is $\alpha$-fractionally log-concave ($\alpha$-FLC) if 
$\log g_{\mu}(z_1^{\alpha}, \dots, z_n^{\alpha})$ is concave, viewed as a function over $\R_{\geq 0}^{n}$. %or, equivalently, $\lambda_{\max}(\corMat_{\mu}(z) \diag(\vec{\alpha})) \leq 1.$
\end{definition}

\subsection{Spectral and Entropic Independence}
\begin{definition}[Correlation matrix] \label{def:corr}
	Let $\mu$ be a probability distribution over $2^{[n]}$. 
Its \textit{correlation matrix} $\corMat_{\mu} \in \R^{n\times n}$  is defined by
\[\corMat_{\mu} (i,j) = \begin{cases} 1 -\P_{\mu}{i} &\text{ if } j=i, \\ \P_{\mu}{j \given i} - \P_{\mu}{j} &\text{ otherwise.}\end{cases} \]
\end{definition}

\begin{definition}[Spectral Independence]
For $\eta \geq 0$, a distribution 
$\mu: 2^{[n]} \to \R_{\geq 0} $ is said to be $\eta$-spectrally independent (at the link $\emptyset$) if
%\[\forall \lambda \in [0,1+\epsilon]^n: g_{\lambda \ast \mu} (z_1, \dots, z_n) = \sum_{S} \mu(S) \lambda^S z^S\]
%is $\eta$-spectral independence, or equivalently, 
\[\lambda_{\max}(\corMat_{\mu}) \leq \eta.\] %\text{ and conditional} \mu' \text{ of  }\mu.\]
\end{definition}

\begin{remark}
The original definition of spectral independence in \cite{ALO20} imposes such a requirement on $\mu$ as well as all of its links. Here, we follow the convention in \cite{anari2021entropic} and use the term spectral independence to refer only to a spectral norm bound on the correlation matrix of $\mu$, with the understanding that in applications, we will require spectral independence of all links of $\mu$ as well.  
\end{remark}

%\Fred{this might differ from other definitions, check.}
\begin{fact}[Remark 70 of \cite{alimohammadi2021fractionally}]\label{fact:si-flc}
A distribution $\mu$ on $\binom{[n]}{k}$ is $\eta$-spectrally independent iff 
\[ \nabla^2 \log g_{\mu}(z_1^{1/\eta},\ldots,z_n^{1/\eta})\Big|_{z = \vec{1}} = (1/\eta)^2 D \corMat_{\mu} - (1/\eta) D \preceq 0 \]
where $D$ is the diagonal matrix with entries $D_{ii} = \P_{\mu}{i}$.

Moreover, $\mu$  is $1/\eta$-FLC iff $\lambda \ast \mu$ is $\eta$-spectrally independent for all external fields $\lambda \in \mathbb{R}_{\ge 0}^n$.
\end{fact}

\begin{lemma}[\cite{ALO20}]\label{lem:si-as-local}
Suppose $P = U_{1 \to k} D_{k \to 1}$ is the transition operator for the $1 \leftrightarrow k$ up-down walk for a distribution $\mu$ on $\binom{[n]}{k}$. Then 
\[ \lambda_2(P) = \frac{\lambda_{\max}(\corMat_\mu)}{k}.  \]
\end{lemma}
\begin{proof}
We include the proof of this result for completeness. 
From the definitions we see that for the vector $d$ with $d_i= \P{i}$, we have that
\[ \corMat_{\mu} = kP - \vec{1} d^T = k \left(P - \frac{1}{k}\vec{1} d^T \right), \]
i.e. $P = \frac{1}{k} \vec{1} d^T + \frac{1}{k} \corMat_\mu$. Observe that $\frac{1}{k} d$ is the stationary distribution of $P$, so $P$ is self-adjoint \cite{levin2017markov} with respect to the inner product $\langle \cdot, \Pi \cdot \rangle$ where $\Pi := \frac{1}{k} \diag(d)$ and $\langle \cdot, \cdot \rangle$ denotes the Euclidean dot product. %, i.e.\ its left-eigenvector with eigenvalue $1$, 
We see by using the variational characterization of eigenvalues that
%and self-adjointness with respect to the inner product $\langle \cdot, D \cdot \rangle$ where $D := \diag(d)$ that 
\begin{equation*}
\lambda_2(P) = \sup_{v} \frac{\langle v, \Pi (P - \frac{1}{k} \vec{1} d^T) v \rangle}{\langle v, \Pi v \rangle} = \sup_{v} \frac{\langle v, \frac{1}{k} \Pi \corMat_\mu v \rangle}{\langle v, \Pi  v \rangle} = \frac{\lambda_{\max}(\corMat_\mu)}{k}. \qedhere
\end{equation*}
\end{proof}

\begin{definition}[Entropic independence]
    \label{def:entropic-independence}
A probability distribution $\mu$ on $\binom{[n]}{k}$ is said to be $C$-entropically independent, for $C\ge 1$, if for all probability distributions $\nu$ on $\binom{[n]}{k}$,
\[ \DKL{\nu D_{k\to 1} \river \mu D_{k\to 1}}\leq \frac{C}{k}\DKL{\nu \river \mu}.  \]
\end{definition}
This is an exact analogue of spectral independence, replacing variance by entropy.
It is also a generalization of subadditivity of entropy which itself is equivalent to a generalized Brascamp-Lieb inequality, see e.g.\ \cite{anari2021entropic,barthe2011correlation,chen2022optimal,blanca2021mixing} for discussion.
\begin{theorem}[\cite{anari2021entropic}]\label{thm:flc-ei}
A distribution $\mu$ on $\binom{[n]}{k}$ is $(1/C)$-FLC if and only if $\lambda \ast \mu$ is $C$-entropically independent for all external fields $\lambda \in \mathbb{R}^{n}_{>0}$ (in particular, all links of $\mu$ are $C$-entropically independent).  
\end{theorem}

\subsection{Functional Inequalities}\label{s:functional-inequalities}
Let $P$ be the transition matrix of an ergodic, reversible Markov chain on a finite set $\Omega$, with (unique) stationary distribution $\mu$. The Dirichlet form of $P$ is defined, for $f,g \colon \Omega \to \R$, by
\[\mathcal{E}_P(f,g) := \frac{1}{2}\sum_{x,y \in \Omega}\mu(x)P(x,y)(f(x) - f(y))(g(x) - g(y)).\]

For later use, we record an equivalent expression for the Dirichlet form of down-up walks. 

\begin{lemma}
\label{lem:dirichlet-down-up}
Let $\mu$ be a distribution on $\Omega := \binom{[n]}{k}$ and let $P$ denote the transition matrix of the $k \leftrightarrow k-1$ down-up walk. Then, for any $f,g : \Omega \to \R$,
\[\mathcal{E}_{P}(f,g) = \E_{S_{k-1} \sim \mu D_{k \to k-1}}{\operatorname{Cov}(f(S), g(S) \mid S_{k-1})}.\]
\end{lemma}

\begin{proof}
For notational convenience, let $\mu_{k-1}:= \mu D_{k \to k-1}$ and $\Omega_{k-1} := \binom{[n]}{k-1}$. We have,
\begin{align*}
    \mathcal{E}_P(f,g) &= \frac{1}{2}\sum_{x,y \in \Omega}\sum_{z \in \Omega_{k-1}}\mu(x)D_{k \to k-1}(x,z)U_{k-1 \to k}(z,y)(f(x)-f(y))(g(x) - g(y))\\
    &= \frac{1}{2}\sum_{z \in \Omega_{k-1}}\sum_{x,y \in \Omega}\mu_{k-1}(z)U_{k-1 \to k}(z,x)U_{k-1 \to k}(z,y)(f(x)-f(y))(g(x) - g(y)) \\
    &= \sum_{z \in \Omega_{k-1}}\mu_{k-1}(z)\frac{\E_{X \sim \mu \mid z, Y \sim \mu \mid z}{(f(X)-f(Y))(g(X)-g(Y))}}{2}\\
    &= \E_{S_{k-1} \sim \mu_{k-1}}{\operatorname{Cov}(f(S), g(S) \mid S_{k-1})}. \qedhere
\end{align*}
\end{proof}

\begin{definition}
The spectral gap or Poincar\'e constant of $P$ is defined to be $\gamma$, where $\gamma$ is the largest value such that for every $f \colon \Omega \to \R$,
\[\gamma \Var_{\mu}{f} \leq \mathcal{E}_P(f,f).\]

The modified log-Sobolev constant of $P$ is defined to be the the largest value $\rho$ such that for every $f \colon \Omega \to \R_{\geq 0}$,
\[\rho \Ent_{\mu}{f} \leq \mathcal{E}_{P}(f,\log f),\]
where $\Ent_{\mu}{f} = \E_{\mu}{f\log f} - \E_\mu{f} \log (\E_\mu{f})$.
\end{definition}
\begin{definition}
Let $P$ be an ergodic Markov chain on a finite state space $\Omega$ and let $\mu$ denote its (unique) stationary distribution.% For any probability distribution $\nu$ on $\Omega$ and 
For any $\varepsilon \in (0,1)$, we define the \emph{$\varepsilon$-total variation mixing time} to be
%\[\tau_{\text{mix}}(\nu,\epsilon) = \min\{t\geq 0 \mid d_{TV}(\nu P^{t}, \mu)\leq \epsilon\},\]
%and
\[\tau_{\text{mix}}(\varepsilon) = \max\set*{\min\set{t\geq 0 \mid d_{TV}(\1_{x}P^t, \mu) \leq \varepsilon}\given x\in \Omega},\]
where $\1_{x}$ is the point mass supported at $x$ and $d_{TV}$ is the total variation distance \cite{cover2012elements}.
\end{definition}

The following relationships between the $\varepsilon$-(total variation) mixing time of $P$, $\tau_{\text{mix}}(\varepsilon)$, and its Poincar\'e and modified log-Sobolev constants is standard (see, e.g.,~\cite{bobkov2006modified,levin2017markov}):
\begin{align}
\label{eqn:mixing-time}
    (\gamma^{-1} - 1)\log\left(\frac{1}{2\varepsilon}\right)\leq \tau_{\text{mix}}(\varepsilon) &\leq \gamma^{-1}\log\left(\frac{1}{\varepsilon}\cdot \frac{1}{\min_{x\in \Omega}\mu(x)}\right),\\
    \tau_{\text{mix}}(\varepsilon) &\leq \rho^{-1}\left(\log\log\left(\frac{1}{\min_{x\in \Omega}\mu(x)}\right) + \log\left(\frac{1}{2\varepsilon^{2}}\right)\right). \nonumber 
\end{align}

\begin{definition}[Approximate tensorization of entropy]
\label{def:approx-tensorization}
A distribution $\mu$ on $\binom{\Omega}{n}$ satisfies approximate tensorization of entropy with constant $C$ if for all probability distributions $\nu$ on $\binom{\Omega}{n}$, 
\[\DKL{\nu D_{n \to n-1} \river \mu D_{n \to n-1}} \leq (1- 1/(Cn))\DKL{\nu\river\mu}.\]
\end{definition}
This is a generalization of the classical definition of approximate tensorization of entropy (see \cite{marton2015logarithmic,caputo2015approximate}), as we observe in the following remark. Explaining this also requires introducing the important concept of \emph{homogenization} which we use throughout this paper:
\begin{definition}[Homogenization]
Let $\mu$ be a distribution on a product space $\Omega' = \Omega_1' \times \dots \times \Omega_n'$. Then $\mu$ can naturally be viewed as a distribution $\mu^{\hom}$ over $\binom{\Omega}{n}$, where $\Omega = \cup_{i=1}^{n}\Omega_i' \times \{i\}$ by identifying $\sigma \in \Omega'$ with the set $\{(\sigma_1,1),(\sigma_2,2),\ldots,(\sigma_n,n)\}$. Note that under this identification, the Glauber dynamics corresponds to the $n \leftrightarrow n - 1$ down-up walk. 
\end{definition}
\begin{remark}
Let $\mu$ be a distribution on a product space and define its homogenization $\mu^{\hom}$ as above.
%$\Omega' = \Omega_1' \times \dots \times \Omega_n'$. By homogenization, $\mu$ can be viewed as a distribution over $\binom{\Omega}{n}$, where $\Omega = \cup_{i=1}^{n}\Omega_i' \times \{i\}$.
In this case, approximate tensorization of entropy with constant $C$ for $\mu^{\hom}$ is equivalent to the assertion that for any positive measurable function $f$,
\[ \Ent_{\mu}{f} \le C \sum_{v = 1}^n \E_{\mu}{\Ent_{v}{f}} \]
%\Fred{does this conflict with lift notation, figure out a consistent notation}
where 
\[ \Ent_k{f} := \Ent_{\mu(\sigma_k = \cdot \mid \sigma_{\sim k})}{f} \]
%$\mu_k$ 
is the entropy functional with respect to the conditional measure of $\sigma_k \in \Omega'_k$ given $\sigma_j \in \Omega'_j$ for all $j \ne k$ and for $\sigma \sim \mu$. 
\end{remark}

It is an immediate consequence of the data processing inequality (see, e.g., \cite{anari2021entropic}) that if $\mu$ satisfies approximate tensorization of entropy with constant $C$, then the $n \leftrightarrow n-1$ down-up walk for $\mu$ has modified log-Sobolev constant at least $1/Cn$.

	%\section{General Inequalities for Spectral Gaps}\label{sec:general-facts}
\section{Universality of Spectral Independence}\label{sec:general-facts}
In this section, we present our key result on the universality of spectral independence, \cref{thm:poincare-to-si}, in the general setting of down-up walks (on pure simplicial complexes). 
\subsection{\texorpdfstring{$k \leftrightarrow k - 1$}{k to k-1} Spectral Gap Implies Spectral Independence}
%The ``local-to-global'' argument allows us to use spectral independence (spectral gap of the $k \leftrightarrow 1$ local walk) at all links to inductively derive bounds on the spectral gap of the $k \leftrightarrow k - 1$ walk. Here, we show a converse which, surprisingly, doesn't seem to have been observed before. We prove spectral gap $1/Ck$ of the $k \leftrightarrow k - 1$ walk lower bounds the spectral gap of the $k \leftrightarrow 1$ walk by $1/C$ i.e. shows $C$-spectral independence (at the empty link). \Fred{discuss somewhere utility of this for proving spectral independence compared to previous works. Maybe discuss how it gives a weak ``completeness'' result for spectral independene framework (establishing polynomial time sampling)}
\begin{theorem}\label{thm:poincare-to-si}
Let $\mu$ be a distribution on $\binom{[n]}{k}$. If the $k \leftrightarrow k - 1$ down-up walk has spectral gap $\frac{1}{Ck}$, then $\mu$ is $C$-spectrally independent.
\end{theorem}
\begin{proof}

From \cref{fact:si-flc}, $C$-spectral independence is equivalent to the inequality $D \corMat_{\mu} \preceq C D$, where $D$ is the diagonal matrix with entries $D_{ii} = \P_{\mu}{i}$.. This is equivalent to showing that for all vectors $v \in \R^{n}$, 
\[ \Var*_{S \sim \mu}{\sum_{i \in S} v_i} \le C \E*_{S \sim \mu}{\sum_{i \in S} v_i^2}. \]
From the definition of the spectral gap and \cref{lem:dirichlet-down-up}, we have for any function $f : \binom{[n]}{k} \to \mathbb{R}$ that
\[ \Var_{S \sim \mu}{f(S)} \le Ck \cdot \E_{S_{k-1} \sim \mu D_{k \to k - 1}}{\Var{f(S) \mid S_{k - 1}}}. \]
Applying this inequality to the function $f(S) = \sum_{i\in S}v_i$, we observe that
\begin{align*} 
\frac{1}{Ck} \Var*_{S \sim \mu}{\sum_{i \in S} v_i} 
&\le \E*_{S_{k-1} \sim \mu D_{k \to k - 1}}{\Var*{\sum_{i \in S} v_i \mid S_{k - 1}}}  \\
&= \E*_{S_{k-1} \sim \mu D_{k \to k - 1}}{\Var*{\sum_{i \in S \setminus S_{k - 1}} v_i \mid S_{k - 1}}}  \\
%&\le C k \E*_{S_{k-1} \sim \mu D_{k \to k - 1}}{\Var*{\sum_{i \in S \setminus S_{k - 1}} v_i }}  \\
&\le \E*_{S_{k-1} \sim \mu D_{k \to k - 1}}{\E*{\sum_{i \in S \setminus S_{k - 1}} v_i^2 \mid S_{k - 1} }} = \frac{1}{k} \E*_{S \sim \mu}{\sum_{i \in S} v_i^2},
\end{align*}
as desired. Here, in the second inequality we have used that $\Var{f(X)} \le \E{f(x)^2}$ and the fact that the sum is over the set $S \setminus S_{k - 1}$ which has size exactly one, and in the last equality we used symmetry. %This proves spectral independence with constant $\eta = C$ as claimed. 
\end{proof}
\begin{remark}
It is well-known that a bounded Poincar\'e constant implies that the largest eigenvalue of the covariance matrix is also bounded. In contrast, \cref{thm:poincare-to-si} shows spectral independence, which is a much stronger property. Spectral independence exactly controls the largest eigenvalue of the \emph{influence} matrix rather than the covariance; reinterpreted in terms of the covariance matrix, it asserts a PSD upper bound not just by a multiple of the identity, but by a multiple of the diagonal matrix $D$ of marginals, which is often much smaller. For example, for the uniform distribution on $\binom{[n]}{k}$, spectral independence tells us that the largest eigenvalue of the covariance matrix is $O(k/n)$, rather than just $O(1)$.
\end{remark}
% For a spin system, a dimension-free Poincar\'e inequality for the Glauber dynamics with constant $C_1$, i.e. spectral gap $1/C_1 n$ for the discrete time dynamics, tells us that for all functions $f$
% \begin{equation}\label{eqn:poincare}
% \Var{f(\sigma)} \le C_1 \sum_{i = 1}^n \E {\Var{f(\sigma} \mid \sigma_{\sim i})}. 
% \end{equation}
% On the other hand, for a spin system the definition of $C_2$-spectral independence\footnote{There might be alternative notational conventions, but this is the one used in EI.} is that for any functions $f_1,\ldots,f_n$ we have
% \begin{equation}\label{eqn:si-spin}
% \Var{\sum_i f_i(\sigma_i)} \le C_2 \sum_{i = 1}^n \E{ f_i(\sigma_i)^2}.
% \end{equation}

% \begin{lemma}\label{lem:poincare-to-si-spin}
% If the Poincare inequality \eqref{eqn:poincare} holds with constant $C_1$, then spectral independence \eqref{eqn:si-spin} holds with $C_2 = C_1$.
% \end{lemma}

\subsection{Trickle-Down of \texorpdfstring{$k \leftrightarrow k - 1$}{k to k-1} Spectral Gap}
As a consequence of the result established in the previous section, we can prove an analogue of the trickle-down theorem (which bounds spectral gaps of $k \leftrightarrow 1$ walks inductively) for the $k \leftrightarrow k - 1$ walk. This follows by combining \cref{thm:poincare-to-si} with Oppenheim's trickle-down theorem \cite{Opp18} and the ``local-to-global'' argument \cite{AL20,KO18}. 
%Prove a general continuity result for spectral gap in set systems, that we later apply to pspin models. % (Fix notation.) %\Fred{define link in preliminaries}
\begin{theorem} \label{thm:continuity}
Let $\mu$ be a distribution on $\binom{[n]}{k}$ with $k \ge 3$ so that for every $i \in [n]$, the $k - 1 \leftrightarrow k - 2$ down-up walk on the link $\mu_i$ has spectral gap at least $1/C(k - 1)$, and such that the $k \leftrightarrow k - 1$ down-up walk on $\mu$ is ergodic. Then the $k \leftrightarrow k - 1$ down-up walk on $\mu$ has  spectral gap at least $1/C'' k$, where $C'' := C \frac{k - 1 - C}{k - 2C}$.
\end{theorem}
\begin{proof}
\cref{thm:poincare-to-si} implies that each link $\mu_i$ is $C$-spectrally independent for every $i$. We claim that $\mu$ is $C'$-spectrally independent for $C' := \frac{C(k - 2)/(k - 1)}{1 - C/(k - 1)}$. Indeed, using \cref{lem:si-as-local} and \cref{thm:oppenheim}:
\[ \lambda_{\max}(\corMat_\mu) = k \lambda_2(P) \le k \frac{(1 - 2/k) C/(k - 1)}{1 - C/(k - 1)} = \frac{C(k - 2)/(k - 1)}{1 - C/(k - 1)} = C'. \]

Next, we use spectral independence to perform a step of the local-to-global argument.
By the law of total variance, for any function $f$ 
\begin{align*}
\Var_{\mu}{f}
&=  \Var*_{i \sim \mu D_{k \to 1}}{\E*_{\mu}{f \mid i}} +  \E*_{i \sim \mu D_{k \to 1}}{\Var*_{\mu}{f \mid i}} \\
&=  \Var*_{\mu D_{k \to 1}}{U_{1 \to k} f} +  \E*_{i \sim \mu D_{k \to 1}}{\Var*_{\mu}{f \mid i}}  \\
&= \E{(U_{1 \to k} (f - \E f))^2} +  \E*_{i \sim \mu D_{k \to 1}}{\Var*_{\mu}{f \mid i}}  \\
&= \E{(f - \E f)(D_{k \to 1}U_{1 \to k} (f - \E f))} +  \E*_{i \sim \mu D_{k \to 1}}{\Var*_{\mu}{f \mid i}}  \\
&\le \lambda_2(D_{k \to 1} U_{1 \to k}) \Var*_{\mu}{f} + \E*_{i \sim \mu D_{k \to 1}}{\Var*_{\mu}{f \mid i}} \\
&\le \frac{C'}{k}  \Var*_{\mu}{f} + \E*_{i \sim \mu D_{k \to 1}}{\Var*_{\mu}{f \mid i}},
\end{align*}
where in the last step we used $\lambda_2(D_{k \to 1} U_{1 \to k}) = \lambda_2( U_{1 \to k} D_{k \to 1}) \le C'/k$  via \cref{lem:si-as-local} and spectral independence.
Hence, rearranging and applying the spectral gap bound for the $k - 1 \leftrightarrow k - 2$ walks of the links,  we have
\begin{align*}
\Var_{\mu}{f}
&\le   \frac{1}{1-C'/k} \E*_{i \sim \mu D_{k \to 1}}{\Var*_{\mu}{f \mid i}} \\
&\le \frac{C (k - 1)}{1-C'/k} \E*_{S_{k - 1} \sim \mu D_{k \to k - 1}}{\Var*_{\mu}{f \mid S_{k - 1}}},
\end{align*}
% i.e.
% \[ \Var_{\mu}{f}  \le  \frac{C}{1 - C'/k} (k - 1) \E*_{S_{k - 1} \sim \mu D_{k \to k - 1}}{\Var*_{\mu}{f \mid S_{k - 1}}}  \]
which bounds the inverse spectral gap of the $k \leftrightarrow k - 1$ walk by $1/C'' k$ for
\[ C'' := \frac{C(1 - 1/k)}{1 - C'/k} = \frac{C(1 - 1/k)}{1 - \frac{C(k - 2)/(k - 1)}{k - Ck/(k - 1)}} = \frac{C(k - 1)(1 - C/(k - 1))}{k - Ck/(k - 1) - C(k - 2)/(k - 1)} = \frac{C(k - 1 - C)}{k - 2C}. \qedhere \]
%and so we have inverse spectral gap
%\[ \frac{C}{1 - C'/k} (1 - 1/k) \]
%i.e.
%\[ a_{N - k} \le \frac{1}{1 - a'_{N - k}/(N - k)} (1 - 1/(N - k)) a_{N - k - 1}\]
%which indeed shows the Poincar\'e constant behaves continuously under the assumption of the lemma.
\end{proof}

	\section{Results for Gibbs Measures on the Hypercube}\label{sec:gibbs}
\paragraph{Notation.} In this section, we consider distributions of the form $\mu(\sigma) \propto \exp(H(\sigma))$ where $H : \{\pm 1\}^n \to \mathbb{R}$. %Following e.g. \cite[Equation (4)]{eldan2018gaussian}, 
Following \cite{eldan2018decomposition}, we identify $H$ with its \emph{multilinear extension} $H : \mathbb{R}^n \to \mathbb{R}$ defined by 
\[ H(x) = \sum_{S \subset [n]} \hat H(S) \prod_{i \in S} x_i \]
where $\hat H(S) := \frac{1}{2^n} \sum_{\sigma \in \{\pm 1\}^n} H(\sigma) \prod_{i \in S} \sigma_i$ is the Fourier transform of $H$ viewed as a function on the hypercube \cite{o2014analysis}. This is also known as the \emph{harmonic extension} since the Laplacian of $H$ vanishes, and for $x \in [-1,1]^n$ it admits an equivalent expression
\begin{equation}\label{eqn:harmonic2}
H(x) = \E_{\sigma \sim \otimes_i \operatorname{Ber}_{\pm}(x_i)}{H(\sigma)} 
\end{equation}
where $\operatorname{Ber}_{\pm}(x)$ denotes the distribution of a random variable valued in $\{\pm 1\}$ with mean $x$.
For $\sigma \in \{\pm 1\}^n$, define
\[ B_j(\sigma) := \partial_j H(\sigma) \]
to be the \emph{cavity field} at site $j$, where $\partial_j H$ is the usual partial derivative applied to the multilinear extension of $H$. Because $H$ is multilinear, the cavity field $B_j$ does not depend on $\sigma_j$. 

Note that in our notational convention, $B_j$ refers to the same object as in \cite{pspin} but $\partial_j$ can differ in sign from their definition. More generally, following \cite{pspin} we define versions of $H$ and $B_j$ for reduced versions of the original system which appear when performing induction. The generalization of $H$ is parameterized by disjoint sets $A,B \subset [n]$ and $\sigma_A \in \{\pm 1\}^A$ and given by 
\[ H_{\sigma_A}^{[A,B]}(\sigma_{(A \cup B)^c}) := \sum_{S \subset B^C} \hat H(S) \prod_{i \in S} \sigma_i.  \]
In the other words, this is the multilinear extension of $H$ evaluated at the vector $(\sigma_A, \sigma_{(A \cup B)^C}, 0_B)$. Similarly, we define
\[ B_j^{[A,B]}(\sigma_{(A \cup B)^c}) := \partial_j H_{\sigma_A}^{[A,B]}(\sigma_{(A \cup B)^c}) \]
where on the left hand side, the dependence on $\sigma_A$ is omitted from the notation for convenience. 
\subsection{Results under the Smoothness Condition} \label{subsec:at-general}
%\Fred{somewhere discuss that this general condition is implied by bounded injective norm e.g. using the argument in \cite{pspin}}

In the main result of this section, we prove that smallness of the Hessian of $H$ implies that the Gibbs measure $\mu(x) \propto \exp(H(x))$ is fractionally log-concave and the Glauber dynamics has $\Omega(1/n)$ spectral gap (i.e.\ relaxation time $O(n)$). Afterwards, in \cref{thm:at-general} we prove that by combining our fractional log-concavity estimate with the entropic independence framework \cite{anari2021entropic}, we obtain strong bounds on the mixing time, MLSI constant, and approximate tensorization constant of $\mu$.
\begin{theorem}\label{thm:poincare bound}
There exist absolute constants $A,B > 0$ for which the following holds. 
 Suppose that $\mu$ is a probability measure with full support on the hypercube $\{\pm 1\}^n$, so $\mu(x) \propto \exp(H(x))$ for some function $H : \{\pm 1\}^n \to \mathbb{R}$. Suppose furthermore that 
\[ \beta := \max_{\sigma \in \{\pm 1\}^n} \|\nabla^2 H(\sigma)\|_{\operatorname{op}} \le A. \]
Then we have that:
\begin{enumerate}
    \item The Poincar\'e constant (spectral gap) for the discrete-time Glauber dynamics on $\mu$ is at least $\frac{1}{(1 + B\beta)n}$.
    \item $\mu^{\hom}$ is $\frac{1}{1 + B\beta}$-fractionally log-concave.
\end{enumerate}
\end{theorem}
This result will be proved by combining the results we established in \cref{sec:general-facts} with an important estimate established in \cite{pspin}, which we now recall. 

%For $k \leq N$ being the number of free (unpinned) vertices and $A $ of size $N-k$ let
We start with some notation.
Let
\[T := \sup_{A, B: A \cap B =\emptyset} \sup_{\sigma_A \in \{\pm 1\}^{|A|}} \sup_{\sigma \in \{\pm 1\}^{N - \abs{A\cup B}}} \norm{\parens{\partial_i B_j^{[A,B]} (\sigma) }_{1\leq i, j \leq N - \abs{A \cup B}}}_{\operatorname{op}} \]
and for any $0 \le k \le N$, let $a_{N - k}$ be the worst-case dimension-free Poincar\'e constant among all subsystems with $|A \cup B| = k$. Precisely, $a_{N - k}$ is defined to be the smallest positive number so that for all $A,B$ disjoint with $|A \cup B| = k$ and all $\sigma_A \in \{\pm 1\}^{A}$, the Glauber dynamics for every measure of the form $\mu^{[A,B]}_{\sigma_A}(\sigma) \propto \exp(H^{[A,B]}_{\sigma_A}(\sigma))$ on $\{\pm 1\}^{(A \cup B)^{c}}$ has spectral gap at least $\frac{1}{a_{N - k}(N - k)}$. 

The following key result from \cite{pspin} relates the values of $a_{N - k}$ between different values of $k$:
%For some fixed constant $C,$$ %(think $r = \beta$)
\begin{lemma}[Proposition 4.1 of \cite{pspin}] \label{lem:inductive}
There exist absolute constants $C,\beta_0 > 0$ for which the following holds. Let $C_r := (Cr)^2 e^{C r} =\Theta_{C}(r^2)$ for $r\leq 1/C$.
%Under the event $\Omega$, 
Suppose that $\beta < \beta_0$ and $\epsilon \in (0,10^{-2})$ are such that $T \le 5 \beta$ and $a_{N - k - 1} < \epsilon/C_{\beta}$. Then %and $\beta < \beta_0$, then
\[ \left(1 - \frac{C \beta^2 e^{C \beta} \max(1,a_{N - k})^2}{\epsilon(N - k)}\right) a_{N - k} \le \left(1 - \frac{1}{N - k}\right) a_{N - k - 1} + \frac{(1 + 4\epsilon)^5}{N - k}. \]
%for some universal constant $C > 0$.
\end{lemma}
%\June{for this equation to be useful we need $ \epsilon > \beta^2.$}
The key difficulty in using this relation to inductively bound the spectral gap is the presence of the term $\max(1, a_{N - k})^2$ which, if large, will make the bound trivial. In \cite{pspin}, this was overcome using a ``continuity argument'' (Section 3 there) which uses properties specific to the $p$-spin model (random $H$). It turns out we can eliminate the need for a specialized continuity argument completely using the general results established in \cref{sec:general-facts}. 

We first check that the assumption on $T$ is satisfied. Actually, it ends up to be equivalent to our definition of $\beta$. 
\begin{lemma} \label{lem:submatrix norm trivial bound}
With the notation above, $T = \beta$.
\end{lemma}
\begin{proof}
Observe that for any point $x \in [-1,1]^n$, by linearity and \eqref{eqn:harmonic2} we have
\[ \nabla^2 H(x) = \E_{\sigma \sim \otimes_i \operatorname{Ber}_{\pm}(x_i)} {\nabla^2 H(\sigma)}. \]
 So by the triangle inequality $\|\nabla^2 H(x)\|_{\operatorname{op}} \le \max_{\sigma \in \{\pm 1\}^n} \|\nabla^2 H(\sigma)\|_{\operatorname{op}}$, and hence
\[ \sup_{x \in [-1,1]^n} \|\nabla^2 H(x)\|_{\operatorname{op}} = \max_{\sigma \in \{\pm 1\}^n} \|\nabla^2 H(\sigma)\|_{\operatorname{op}}. \]
From the definition, it's clear that $T \ge \beta$ and since $H^{[A,B]}(\sigma_{(A \cup B)^c})$ is the multilinear extension of $H$ evaluated at the vector $(\sigma_A, \sigma_{(A \cup B)^C}, 0_B)$, it follows from the above argument that $T \le \beta$.
\end{proof}

% \begin{theorem}[cite] \label{thm:poincare bound}
% If $\Psi_A$ holds then for any pinning $ \sigma_A$, the Poincare constant of the conditional distribution on $\bar{A} $ of size $N-k$ is $ a_A \leq 1 +\beta_k $ 
% \June{Arka's paper use $ \epsilon = \beta^{4/3}$ (mid of page 19)}
% \end{theorem}

% \begin{corollary}
% Under $\Psi_A,$ for any pinning $\sigma_A,$ the conditional distribution on $\bar{A}$ is $\gamma_A:= (1+\beta_k)$-entropic independence. 
% \end{corollary}

%supporting lemmas
\begin{proof}[Proof of \cref{thm:poincare bound}]
First, we observe that if we establish conclusion (1) concerning the spectral gap, it will automatically imply conclusion (2) regarding fractional log-concavity via Theorem~\ref{thm:poincare-to-si}. Indeed, by \cref{fact:si-flc}, we know that establishing $\mu^{\hom}$ is $\frac{1}{1 + B \beta}$-FLC is equivalent to showing that for all external fields $\lambda$, the tilted measure $\lambda \ast \mu^{\hom}$ is $(1 + B \beta)$-spectrally independent. Note that tilting $\mu^{\hom}$ is equivalent to first changing the degree-one part of $H$, and then homogenizing the resulting measure. Furthermore, changing the degree-one part of $H$ does not change $\nabla^2 H$ and hence the assumption of this theorem is invariant to arbitrary tilts by external fields. 
Since the Glauber dynamics is the $n \leftrightarrow n - 1$ down-up walk on $\mu^{\hom}$, 
it therefore suffices to prove that the spectral gap of the Glauber dynamics is at least $1/(1 + B \beta)n$, which is precisely what conclusion (1) says.
%Finally, we observe that because the assumption of this theorem is invariant to tilts (tilting $\mu^{\hom}$ by an external field is equivalent to changing the degree-1 part of $H$ and then homogenizing, so it does not affect $\nabla^2 H$), so 
%Applying conclusion (1) implies that the spectral gap is indeed lower bounded by $1/(1 + B \alpha)n$ and concludes the argument.

It remains to prove conclusion (1). 
%We need only to prove the theorem for the case $t = N$ (no pinning).
We bound $a_{N-k}$ by induction on $(N-k).$ %Let $\epsilon = \beta^2\leq 1/2.$ 
The induction hypothesis is $ a_{N-k} \leq 1 + \delta$ with $\delta \leq 1/2$ to be chosen later. For base cases, we have the bound when $N - k \le 2$ because in this case the measure is an Ising model (on one or two sites), so the desired bound follows from Theorem 1 of \cite{eldan2021spectral}.
%The base case is $a_1 = 1.$ (we might need more base case)

Let $k \leq N-3$ and
assume that $a_{N-k-1} \leq 1+  \delta$. We will prove $a_{N-k} \leq 1 + \delta.$

By \cref{thm:continuity},
\[a_{N-k} \leq a_{N-k-1}  \frac{N-k-1-a_{N-k-1}}{N-k-2 a_{N-k-1}} = a_{N-k-1} T_{N-k-1}\]
with $$T_{N-k-1} := \frac{1}{2} + \frac{(N-k)/2-1}{N-k-1-a_{N-k-1}} \leq \frac{1}{2} + \frac{3/2}{2-3/2} = 7/2$$ for $k \leq N-3.$
%\[a_{N-k} \leq \frac{1 + 1/(N-k-2)}{1 - a_{N-k-1}/(N-k-2)} a_{N-k-1} \leq \frac{1+1/2}{1- (1+\varepsilon)/2} a_{N-k-1} \leq T_{N-k-1} a_{N-k-1} \leq T_{N-k-1}(1+\varepsilon).\]
%with $T_{N-k-1} = \frac{1 +1/(N-k-2) }{1- (1+\varepsilon)/(N-k-2)}.$

Let $\epsilon = \delta/10$. Below, we will choose $\delta = \omega_{\beta \to 0}(\beta^2)$. In particular, for $\beta_0$ sufficiently small, we have
$$a_{N-k-1} \leq 1+ \delta \leq \frac{\epsilon}{C_{\beta}}.$$
Hence, by
\cref{lem:inductive} we have
\[\left(1 - \frac{T_{N-k-1}^2 C\beta^2 e^{C\beta } (1+\delta)^2 }{\epsilon(N-k)}\right) a_{N-k} \leq \left(1-\frac{1}{N-k}\right) a_{N-k-1} +\frac{(1+4\epsilon)^5}{N-k}\]
thus
\[a_{N-k} \leq \left(1-\frac{1}{N-k}\right) a_{N-k-1} + a_{N-k} \frac{T_{N-k-1}^2 C\beta^2 e^{C\beta } (1+\delta)^2 }{\epsilon(N-k)}  +\frac{(1+4\epsilon)^5}{N-k}  \]
Since $a_{N-k} \leq a_{N-k-1}T_{N-k-1}$, we can bound the second term by $\frac{T_{N-k-1}^3 C \beta^2 e^{C\beta} (1+\delta)^2}{\epsilon (N-k)} a_{N-k-1} .$ Assuming $ \epsilon \geq 3 \delta^{-1} T_{N-k-1}^3 C\beta^2 e^{C\beta}$ we have
\begin{align*}
    a_{N-k} &\leq \left(1-\frac{1}{N-k} + \frac{\delta}{3(N-k)}\right) a_{N-k-1} + \frac{(1+4\epsilon)^5}{N-k}\\
    &\leq \left(1 + \frac{\delta/3 -1}{N-k}\right) (1+ \delta)  +\frac{(1+4\epsilon)^5}{N-k}\\
    &\leq (1+\delta) + \frac{1}{N-k} \parens*{(\delta/3 -1)(\delta +1) + (1+4\epsilon)^5}.
\end{align*}

Recall that $\epsilon = \delta/10$. Substituting this in the second term, we can verify that for all $\epsilon = [0,0.01]$, $g(\epsilon) = (10\epsilon/3-1)(10\epsilon+1) + (1+4\epsilon)^5 < 0,$ so that the induction holds provided that $\epsilon \geq 3\delta^{-1}T^{3}_{N-k-1}C\beta^2 e^{C\beta}$ (i.e. $10\epsilon^2 > 3T^3_{N-k-1}C\beta^2 e^{C\beta}$) and $\delta = \omega_{\beta \to 0}(\beta^2)$. Since $T_{N-k-1} \leq 7/2$ under the inductive hypothesis, it is readily seen that taking $\epsilon = \Omega_{C}(\beta)$ suffices for the induction to hold.   

%For $ k \leq N-4$ we can bound $T_{N-k-1}$ by $6$ so the only condition we need is $\epsilon^2 \geq3\times 6^3 \times C\beta^2 e^{C\beta} $ which roughly requires $\beta \leq O(\epsilon).$ We can choose $\epsilon = \theta(\beta).$ 
Finally, the bound on the Poincare constant is $ a_N \leq 1+\delta = 1 + \Theta_{C}(\beta).$ \qedhere

% \June{attempt to improve dependency of $\beta$ on $\epsilon$
% Set $\varepsilon =\epsilon^x$ for some $x \in (0,1).$ We only need to choose $\varepsilon$ so that $g(\epsilon) = (\epsilon^x/3 -1)(\epsilon^x +1) + (1+4\epsilon)^5\leq 0 $ for all $\epsilon := \varepsilon^r \in [0, \epsilon_{\max}:=0.8]. $
% \June{this holds in Wolfram-alpha for $x = 1/30.$ The smaller $x$ is the better for the dependency of $\beta$ vs. $\epsilon.$}

% Check that $g$ is convex in $[0,1].$
% Note that $g(0) =0$ and $g(0.8)< 0$ gives the proof.

% %We first show that $g$ is convex in the interval. Since $ g(0) = 0$ and $g(\epsilon_{\max}) $ %Clearly, $g'' = \frac{1}{3} + 320 r(r-1) (1+4$

% \June{we need $(1+4\epsilon)^5 \leq 1+ \varepsilon .$ Even for $\varepsilon =1,$ this implies $\epsilon \leq (\sqrt{2}-1)/4 \approx 0.037.$ So this constant $4$ is a bottleneck. If we can get $C=1$ and $T_{N-k-1}\approx 1$ then $\beta \approx \sqrt{\epsilon} \approx \sqrt{0.037} = 0.19$ }
% }
\end{proof}

\begin{theorem}\label{thm:at-general}
Suppose $\mu$ and $\beta $ satisfy the same assumptions as \cref{thm:poincare bound}. Then, with the notation there, we have:
\begin{enumerate}
    \item $\mu$ (equivalently $\mu^{\hom}$) satisfies approximate tensorization of entropy with constant $C = O(n^{B\beta})$.
    \item The discrete-time Glauber dynamics on $\mu$ satisfy the Modified Log-Sobolev Inequality (MLSI) with constant $\rho = \Omega(1/n^{1 + B\beta})$. 
    \item The discrete-time Glauber dynamics on $\mu$ satisfy 
    \[ \tau_{\text{mix}}(\epsilon) = O(n^{1 + B\beta}(\log\log(1/\min_{\sigma} \mu(\sigma)) + \log(1/\epsilon)). \]
\end{enumerate}
\end{theorem}

The proof of this theorem requires the following intermediate lemma showing approximate entropy tensorization for a system satisfying the conditions of \cref{thm:at-general} under pinning with $N-k= O(1)$ free (unpinned) spins. \cite[Lemma 2.2]{caputo14} showed this result only for 2-spin systems, but the proof applies more generally. We repeat the proof for the sake of completeness. The details are in \cref{app:deferred-proofs}. 

\begin{lemma}\label{cor:at for constantly many free spins}
Let $\mu$ be a distribution satisfying the assumptions of \cref{thm:at-general}.
Any pinning of $\mu$ where the number of free spins $N-k$ is constant ($N-k=O(1)$) has approximate entropy tensorization with constant $ C= \exp(O(\beta)).$
\end{lemma}

\begin{proof}[Proof of \cref{thm:at-general}]
The second and third conclusion follow directly from approximate tensorization of entropy (see \cref{sec:preliminaries}). 

Let $\alpha = 1/(1+B\beta)$. By Theorem 5 of \cite{anari2021entropic}, we have that if the measure $\mu^{\hom}$ is $\alpha$-FLC %for some $\alpha \in [1/2,1]$, 
then for any measure $\nu$ absolute continuous with respect to $\mu$ and $k_0 = \ceil{1/\alpha}$,
\[ \DKL{\nu D_{n \to (n - k_0)} \river \mu^{\hom} D_{n \to (n - k_0)}} \le (1 - \kappa) \DKL{\nu \river \mu^{\hom}} \]
where
\[ \kappa = \frac{(3 - 1/\alpha)^{1/\alpha - \lceil 1/\alpha \rceil} \prod_{i = 0}^{\lceil 1/\alpha \rceil - 1} (2 - i)}{(n + 1)^{1/\alpha}}\geq n^{1/\alpha-1} = n^{B \beta}  \]
This is equivalent to block approximate entropy tensorization with block size $k_0$ \cite[Eq.(1.5)]{blanca2021mixing} i.e.
\[\frac{k_0}{n}\Ent_{\mu}{f} \le n^{B\beta} \frac{1}{\binom{n}{k_0}}\sum_{S\in \binom{[n]}{k_0}} \E_{\mu}{\Ent_{S}{f}}.\]
Since $|S| = k_0 = O(1+B\beta) = O(1)$ and $\beta = O(1)$, combining this with
 \cref{cor:at for constantly many free spins} gives
\[\frac{1}{n}\Ent_{\mu}{f} \le n^{B\beta} \frac{\exp(O(\beta) )}{k_0\binom{n}{k_0}}\sum_{S\in \binom{[n]}{k_0}} \sum_{v\in S} \E_{\mu}{\Ent_{v}{f}} = n^{B\beta} \frac{\exp(O(\beta) )}{n}\sum_{v\in [n]} \E_{\mu}{\Ent_{v}{f}},\]
as required.
%where we use the fact that $k_0 = O(1)$ and for $\Ent_S{f}$ is the entropy of $f$ on a system with $\abs{S} = k_0$ many free (unpinned) spins and $\beta = O(1).$
\end{proof}

\subsection{Results for the \texorpdfstring{$p$-Spin}{p-Spin} Model}
In this subsection, we prove \cref{thm:mix p spin main}. 
% \begin{theorem}\label{thm:mix p spin main}
%     Let $\mu$ be the $p$-spin model i.e. $\mu(\sigma)\propto \exp(H(\sigma)) $
%     with \[ H(\sigma) = \sum_{p = 2}^{\infty} \frac{\beta_p}{n^{(p - 1)/2}} \sum_{1 \le i_1, \ldots, \le i_p \le n} g_{i_1 \cdots i_p} \sigma_{i_1 \cdots i_p} + \sum_i h_i \sigma_i \]
% where the sum ranges over distinct $i_1,\ldots,i_p$, each of the coefficients $g_{i_1 \cdots i_p}$ is i.i.d. Gaussian, and we allow the external fields $h_i$ are arbitrary, with possible dependency on $g.$
% There exists absolute constant $A$ for with the following holds.
% Suppose $\beta := \sum_{p> 2}\sqrt{p^3 \log p} \beta_p \leq A.$ 
% With probability $\geq 1 -\exp(-\theta(N))$ over the randomness of the Gaussian interaction terms $g$,
% \begin{enumerate}
%     \item $\mu$ (equivalently $\mu^{\hom}$) satisfies approximate tensorization of entropy with constant $C = O(1)$.
%     \item The discrete-time Glauber dynamics on $\mu$ satisfy the Modified Log-Sobolev Inequality (MLSI) with constant $\rho = \Omega(1/n)$. 
%     \item The discrete-time Glauber dynamics on $\mu$ satisfy 
%     \[ \tau_{\text{mix}}(\epsilon) = O(n^{1}(\log\log(1/\min_{\sigma} \mu(\sigma)) + \log(1/2\epsilon^2)). \]
% \end{enumerate}
% \end{theorem}
As mentioned in the introduction, we will need to rely on an average case local to global argument. Specifically, we will need the following theorem, which is a slight modification of {\cite[Theorem 20]{avg-local-global}} and follows from the same proof.
\begin{theorem}\label{thm:avg local to global entropy contraction}%{\cite[Theorem 20]{avg-local-global}} 
Consider a distribution $\mu: \binom{[n]}{k} \to \R_{\geq 0}.$
	Suppose that for every set $T$ of size $\leq k-2$,
 $\mu_T$ is $(k-\card{T})(1-\rho(T))$-entropically independent i.e. %$D_{(k-\abs{T}\to 1}$ wrt $\mu_T$ contracts KL-divergence in terms of a factor parameterized by $\rho(T)$:
	\[
		\DKL{\nu D_{k-\card{T}\to 1} \river \mu_T D_{k-\card{T}\to 1}}\leq (1-\rho(T))\DKL{\nu\river \mu_T}.
	\]

 Suppose also that there exist constants $k_0\geq 2$ and $C(k_0)$ such that for all $T$ with $\abs{T} \geq k-k_0$, $\mu_T$ satisfies approximate entropy tensorization with constant $C(k_0)$. 
 
 Finally, for a set $T$ of size $\geq k-1$, define the harmonic mean
		\[ \gamma_T:=\E*_{e_1,\dots,e_{\card{T}}~\text{uniformly random permutation of }T}{\parens*{\rho(\emptyset)\rho(\set{e_1})\rho(\set{e_1,e_2})\cdots \rho(\set{e_1,\dots,e_{k-k_0}})}^{-1}}^{-1}. \]

	Then the operator $D_{k\to (k-1)}$ satisfies
	\[ \DKL{\nu D_{k\to(k-1)} \river \mu D_{k\to(k-1)}} \leq (1-\kappa) \DKL{\nu \river \mu}, \]
	with
	\[ \kappa:=C(k_0)^{-1}\min\set*{\gamma_T\given T\in \binom{[n]}{k-1}}. \]
\end{theorem}

{\bf Notation. }For each $k$, let $A_k$ be a uniformly random set of $k$ spins\footnote{Note that for this subsection, $k$ is the number of free (unpinned) spins, unlike in \cref{subsec:at-general}.}. We will bound the norm of the tensor associated to $A_k$ when $A_k^c$ is fixed according to $\sigma$. It will be more convenient to work with the normalization $\sigma \in \{\pm 1\}^{N}/\sqrt{N}$. For a sequence $S$ of coordinates and a spin configuration $\sigma$ on a set of coordinates containing $S$, we denote $\sigma_S = \prod_{i\in S}\sigma_i$. We write $|S|$ for the length of $S$, and write $S\subset A$ if all entries of $S$ are in $A$. For sequences $S,T$ of coordinates, we denote by $g_{(S,T)}$ the entry of the disorder indexed by $(S,T)$; and $g_{\{S,T\}}$ the sum of all entries whose index is the union of $S$ and $T$ (with the same order of elements within $S$ and $T$). For a given realisation of $A_k$, disjoint subsets $A',B'$ of $A_k$ and $\sigma' \in \{\pm 1\}^{A_k \setminus B'}/\sqrt{N}$, we define the matrix $\Delta^{[A',B']}$, with rows and columns indexed by $A_k \setminus (A' \cup B')$ (here, we are suppressing dependence on $\sigma, \sigma', A_k$): 
\begin{equation}
\label{eqn:Delta-matrix}
\Delta^{[A',B']}_{i,j} = \sum_{p\geq 2}\frac{1}{\sqrt{N}} \beta_{p} \sum_{\substack{s+s'=p-2}} \sum_{\substack{S\subset A_k^c,|S|=s,\\ S'\subset A_k\setminus B', |S'|=s'}}g_{\{i,j,S,S'\}}\sigma_{S}\sigma'_{S'},
\end{equation}
i.e.~$\partial_{i} B_{j}^{[A',B']}(\sigma') = \Delta^{[A',B']}_{i,j}$.

% For $A',B'\subset A_k$, and $\sigma'\in \{\pm 1\}^{A_k\setminus B'}$, 
% define $\Delta^{[A',B']}$ to be a matrix indexed by $A_k\setminus (A'\cup B')$ with $\Delta^{[A',B']}(i,j)=\sum_{S\subset A_k^c,S'\subset A_k\setminus B'}g_{\{i,j,S,S'\}}\sigma_S\sigma'_{S'}$. As such, $(\partial_iB_j^{[A',B']}(\sigma'))=\Delta^{[A',B']}$. 

%Define $\mathcal{S}_{s,s'}(A_k)$ to be the collection of subsets $S\subset A_k^c$ and $S'\subset A_k$ with $|S|=s,|S'|=s'$. 
%Let $\Theta^{[A',B']}(i,j)=\sum_{S'\subset A_k \setminus B'}g_{\{i,j,S'\}}\sigma'_{S'}$. Note that $(\partial_iB_j^{[A',B']}(\sigma'))=\Theta^{[A',B']}+\Delta^{[A',B']}$. 
%We will first bound the expectation of the norm (under the randomness of $A_k$ and with $\sigma$ fixed). 

%For $p$-spin models, we show that the approximate entropy tensorization factor is $O(1).$ It is enough to prove the following lemma. %The desired contraction factor of $D_{n \to (n-1)}$ for $\mu$  
%Combining \cref{lemma:p spin main} with \cref{thm:avg local to global entropy contraction,cor:at for constantly many free spins} immediately imply the desired result. 

The key additional ingredient we need to prove \cref{thm:mix p spin main} using \cref{thm:avg local to global entropy contraction} is the following estimate. 

\begin{proposition} \label{lemma:p spin main}
There exists a constant $C > 0$ such that, with notation as in the proof of \cref{thm:mix p spin main}, with probability at least $1-\exp(-\Theta(N))$ over the choice of the Hamiltonian $H=H(g)$,
\begin{align}
    &\sup_{I\in [N], \sigma \in \{\pm 1\}^{[N]\setminus \{I\}}}\mathbb{E}_{\pi \in \operatorname{Sym}(N-\{I\})}\left[\exp\left(\frac{B \sum_{N\geq k>k_0}\sup_{A',B'\subset A_k}\sup_{\sigma'\in \{\pm 1\}^{A_k\setminus B'}}\|\Delta^{[A',B']}\|_{\operatorname{op}}}{k}\right)\right] \leq C.
\end{align}
Here $A_k$ is the union of $\{I\}$ and the last $k-1$ elements according to the permutation $\pi$ on $[N]\setminus \{I\}$.
% Assume that there exists $\alpha_p$ with $\sum \alpha_p^{-1}<\infty$ and $\sum \alpha_p\beta_p^2 < \infty$. We have with high probability over $G$ that there exists constant $k_0 >0$ s.t. $\forall k > k_0$ 
% \begin{equation}
% \sup_{\sigma,i}\mathbb{E}_{A_k\subset [N],i\in A_k}\left[\sup_{A',B'\subset A_k}\sup_{\sigma'}\|\Delta^{[A',B']}\|\right] \le (\sum_p \beta_p^2/N)^{1/2}.
% \end{equation}
% %Similarly,
% %\begin{equation}
% %\sup_{\sigma}\mathbb{E}_{A_k}\left[\sup_{A',B'\subset A_k}\sup_{\sigma'_{A'}}\sup_{\sigma'_{A_k\setminus(A'\cup B')}}\|\Delta^{[A',B']}\|^{h}\right] \le Ch^{1/2}(\sum_p \beta_p^2/N)^{\min(h,k)/2}.
% %\end{equation}
% As a corollary, with probability $1-\exp(-cN)$, 
% \begin{align*}
% &\sup_{\sigma,i} \mathbb{E}_\pi[\exp(\beta \sum_{k>k_0} \sup_{A',B'\subset A_k}\sup_{\sigma'}\|\Delta^{[A',B']}+\Theta^{[A',B']}\|/k)] \\
% &\le N^{-1/2}N^{\epsilon}+\exp(N^{-1/2})+\exp(\beta \sum_{k>k_0}  \sqrt{k/N}\log(CN/k)/k)\\
% &<C.
% \end{align*}
\end{proposition}
The proof of \cref{lemma:p spin main} is presented at the end of this section. 

Given \cref{thm:at-general}, \cref{thm:avg local to global entropy contraction} and \cref{lemma:p spin main}, the proof of \cref{thm:mix p spin main} is immediate. 

\begin{proof}[Proof of \cref{thm:mix p spin main}]
%By \cite[Lemma 6.1]{pspin}, with probability $1-\exp(-N),$ $\norm{\nabla^2 H}_{OP}\leq \beta.$

Let $B$ be the constant appearing in \cref{thm:poincare bound}.
For $A\subseteq [n]$ and $\sigma_A \in \set{\pm 1}^A$, as in \cref{subsec:at-general}, let $H_{\sigma_A}$ be the Hamiltonian of the subsystem where the spins in $A$ are pinned according $\sigma_A.$ 

%We first establish entropy contraction of $D_{k\to \ell}$ with $\ell = k- \ceil{k_0}.$ 
We want to establish approximate tensorization of entropy, which is the same as entropy contraction of $D_{N\to (N-1)}$ for $\mu^{\hom}.$
The set $T$ of size $N-1$ in \cref{thm:avg local to global entropy contraction} corresponds to a tuple $ (i, \sigma)\in [N] \times \set{\pm 1}^{[N] \setminus \set{i}}$. %with $T \equiv [N] \setminus \set{i}.$ 
Each permutation $e_1, \dots, e_{\abs{T}}$ of $T$ corresponds to a permutation $\pi $ of  $[N] \setminus \set{i}$. For each $ k$, $\rho(\set{e_j\given j \leq N-k}) $ corresponds to the KL-divergence contraction of the  $D_{k \to 1}$ operator wrt $\mu_{\sigma_{A_{\pi,k}^c}}$ where $A_{\pi,k}^c =\set{e_j\given j \leq N-k}.$ 
%By \cref{lem:submatrix norm trivial bound}, $\norm{\nabla H_{\sigma_{A_{\pi,k}^c}}} \leq \beta<A.$

By \cref{thm:poincare bound}, 
$\mu_{\sigma_{A_{\pi,k}^c}}$ is $\frac{1}{1+ B  T_{\sigma_{A_{\pi,k}^c}} }$-fractionally log concave so by \cref{thm:flc-ei} the $D_{k \to 1}$ operator contracts KL-divergence by  $$\rho (\sigma_{A_{\pi,k}^c}) = 1- \frac{1+ B T_{\sigma_{A_{\pi,k}^c}} }{k}$$ with $T_{\sigma_{A_{\pi,k}^c}} = \sup_{A', B' \subseteq A_{\pi, k} } \|\Delta^{[A',B']}\|_{\operatorname{op}}.$
%$\rho(\sigma_A) = 1- \frac{1+ B \norm{\nabla^2 H_{\sigma_{A_k^c}} }_{OP}}{n-\abs{A}}.$

Let $k_0$ be a constant to be chosen later.
Then, by \cref{cor:at for constantly many free spins} and \cref{thm:avg local to global entropy contraction}
\begin{align*}
  &\max \gamma_T^{-1} \equiv \sup_{\sigma, i} \gamma_{\sigma,i}^{-1} \\
  &= \exp(\beta k_0) \sup_{\sigma, i} \mathbb{E}_{\pi}\left[\prod_{N \geq k > k_0} \left(1- \frac{1+ B  T_{\sigma_{A_{\pi,k}^c}}}{k}\right)^{-1} \right]  \\
  &\leq  \exp(\beta k_0) \sup_{\sigma, i}\mathbb{E}_\pi\left[\exp\left(\sum_{N\geq k>k_0} \frac{1}{k} + \frac{B \sum_{N\geq k>k_0}\sup_{A',B'\subset A_k}\sup_{\sigma' \in \{\pm 1\}^{A_k \setminus B'}}\|\Delta^{[A',B']}\|_{\operatorname{op}}}{k}\right)\right]\\
  &\leq N \exp(\beta k_0 + C)
  \end{align*}
where the last inequality holds with probability $\geq 1-\exp(\Theta(N))$ by \cref{lemma:p spin main}, and $k_0$ and $C$ are chosen according to \cref{lemma:p spin main}. %$C$ is the constant from .
%Taking union bound over the two good events gives the desired result.
\end{proof}

Finally, we prove \cref{lemma:p spin main}.

%\subsection{Proof\texorpdfstring{ of \cref{lemma:p spin main}}{}}
\begin{proof}[Proof of \cref{lemma:p spin main}]
Throughout we fix an index $I \in [N]$ according to \cref{lemma:p spin main} and let $A_k$ denote a random set of size $k$ which has the distribution of $I$ together with a uniformly random subset of $[N]\setminus \{I\}$ of size $k-1$. We also fix $\sigma \in \{\pm 1\}^{N}/\sqrt{N}$. %(it will be more convenient to work with this normalization). 
\begin{comment}
By the same argument as \cite[Theorem 6.1]{pspin}, we can write
\[\sigma_i \partial_i B_j^{[A,B]} (\sigma) = \sum_{p=2}^{\infty} \frac{\beta_p}{\sqrt{N}} G_p (\tau, \dots, \tau, \textbf{e}_i, \textbf{e}_j)\]
with $\tau_i = \begin{cases} 0 &\text{ if } i\in B \\ \sigma_i/\sqrt{N} &\text{ else}\end{cases}$
thus for $ \textbf{u}\in \mathcal{S}_{N - \abs{A\cup B}-1 }, $ we can embed $\textbf{u}$ into $\textbf{w}\in\R^N$ by padding with zeros, and
\[u^\intercal ( \partial_i B_j^{[A,B]} (\sigma)) u = \sum_{i,j} u_i \sigma_i \parens*{ \sum_{p=2}^{\infty} \frac{\beta_p}{\sqrt{N}} G_p (\tau, \dots, \tau, \textbf{e}_i, \textbf{e}_j)} u_j = \sum_{p=2}^{\infty} \frac{\beta_p}{\sqrt{N}} G_p( \tau, \cdots, \tau, \sigma \ast \textbf{w}, \textbf{w})  \]
where $(\sigma \ast w)_i = \sigma_i w_i.$
\end{comment}

% For a given realisation of $A_k$, disjoint subsets $A',B'$ of $A_k$ and $\sigma' \in \{\pm 1\}^{A_k \setminus B'}/\sqrt{N}$, note that the matrix $\Delta^{[A',B']}$ (here, we are suppressing dependence on $\sigma, \sigma', A_k$) is defined so that for $u \in \R^{A_k \setminus (A' \cup B')}$, 
% \[
% u^T \Delta^{[A',B']}u = \sum_{i,j\in A_k\setminus(A'\cup B')}u_iu_j\left(\sum_p\frac{1}{\sqrt{N}} \beta_{p} \sum_{\substack{s+s'=p-2}} \sum_{\substack{S\subset A_k^c,|S|=s,\\ S'\subset A_k\setminus B', |S'|=s'}}g_{\{i,j,S,S'\}}\sigma_{S}\sigma'_{S'}\right),
% \]
Recall the definition of the matrix $\Delta^{[A',B']}$ from \cref{eqn:Delta-matrix}. Note that $g_{\{i,j,S,S'\}}$ is a sum of $O\left(\binom{p}{2}\binom{p-2}{s'}\right)$ independent standard Gaussians (the union $S \cup S' \cup \{i,j\}$ has size at most $p$; given this union, we can uniquely recover $S$ as the part of the union belonging to $A_{k}^{c}$; we can then choose the indices $i,j$ from the remaining elements, and together with size constraints, this determines $S'$ up to constantly many choices. Finally, there are at most $\binom{p-2}{s'}$ ways to choose the positions of the indices corresponding to $S'$). We write (suppressing the dependence on $\sigma, \sigma'$):
\[\Delta^{[A',B']} = \Delta_{0}^{[A',B']} + \Delta_{1}^{[A',B']},\]
where $\Delta_{0}^{[A',B']}$ includes the terms with $s' = 0$ and $\Delta_{1}^{[A',B']}$ consists of the remaining terms (i.e.~those with $s'\geq 1$). We also denote $\Delta_{0,p}^{[A',B']}$ and $\Delta_{1,p}^{[A',B']}$ to denote the parts of these matrices stratified by $p\geq 2$.

For notational convenience, given $p\geq 2$, $\sigma \in \{\pm 1\}^N/\sqrt{N}$, $S'\subset A_k\setminus (A' \cup B')$ of size $s' \geq 0$ and $i,j \in A_k \setminus B'$, we define $$X^{i,j,S'}_{\sigma, p} :=\sum_{S\subset A_k^c}\frac{1}{\sqrt{N}}g_{\{i,j,S',S\}}\sigma_S.$$ 
Note that these are independent (over the choice of $S',i,j$) Gaussians with variance $O\left(p^{2}\binom{p-2}{s'}N^{-1}\right)$. 

\begin{claim}
\label{claim:s'ge1}
With notation as above,
\[\mathbb{P}_{g}\left[\forall p, \sup_{\sigma, A_{k}, A', B', S', \sigma'}\|\Delta_{1,p}^{[A',B']}\|_{\operatorname{op}} \geq C\beta_{p}\sqrt{\log{p}} \sum_{1\leq s'\le p-2}\binom{p}{s'}^{1/2}(k/N)^{s'/2} \right] \leq \exp(-1000N).\]
\end{claim}
\begin{proof}
The proof of this claim is standard. Note that the $(i,j)^{th}$ entry of $\beta_{p}^{-1}\Delta_{1,p}^{[A',B']}$ is $\sum_{S'\subseteq A_k \setminus B'}X_{\sigma}^{i,j,S'}\sigma_{S'}$, which is a Gaussian of mean $0$ and variance $O(p^2\binom{p-2}{s'} N^{-1} k^{s'}/N^{s'})$. Moreover, the entries of the matrix are independent. Therefore, by the concentration of the norm of random subgaussian matrices (e.g.~\cite[Theorem~4.4.5]{vershynin2018high}), we have that 
\[\|\Delta_{1,p}^{[A',B']}\|_{\operatorname{op}} \gtrsim \beta_{p}p\sqrt{\log{p}}\sum_{1\le s'\le p-2}\binom{p}{s'}^{1/2}(k/N)^{s'/2}\] with probability at most $\exp(-100N\log{p})$. 
We can then comfortably take the union bound over all choices appearing in the supremum. 
\end{proof}
Furthermore, observe that
\begin{align*}
    \sum_{p\geq 2} \sum_{k = k_0}^{N} \frac{1}{k}\beta_{p}p\sqrt{\log{p}}\sum_{1\le s'\le p-2}\binom{p}{s'}^{1/2}(k/N)^{s'/2} \le \sum_{p\ge 2} \beta_p \sqrt {p^3\log p} 2^{p/2},
\end{align*}
which shows that on the event in \cref{claim:s'ge1}, for all choices of $\sigma, I$ and for all permutations $\pi$,
\begin{equation}
\label{eqn:Delta-1-bound}
    \exp\left(\frac{B \sum_{N\geq k>k_0}\sup_{A',B'\subset A_k}\sup_{\sigma'\in \{\pm 1\}^{A_k\setminus B'}}\|\Delta^{[A',B']}_{1}\|_{\operatorname{op}}}{k}\right) \lesssim 1.
\end{equation}

It therefore remains to prove \cref{lemma:p spin main} with $\Delta$ replaced by $\Delta_0$. Note that so far, in our analysis of $\Delta_{1}^{[A',B']}$, we did not need to use the randomness in the choice of $A_{k}$. This will be crucial in controlling $\|\Delta_{0}^{[A',B']}\|_{\operatorname{op}}$. Observe that $\Delta_{0}^{[A',B']}$ does not depend on $\sigma'$. Given $\sigma$ and $I$, we say that $A_{k}$ is \emph{bad} if $\sup_{p}\sup_{A',B'\subseteq A_{k}}\|\Delta_{0,p}^{[A',B']}\|_{\operatorname{op}}/(\beta_{p}p\sqrt{\log{p}}) \geq C(k/N)^{1/2 - \alpha}$, where $\alpha < 1/2$ is a constant (for instance, $\alpha = 1/4$ is sufficient for us). 

\begin{claim}
\label{claim:s'=0}
    For any $\sigma$ and $I$, the probability that the number of bad $A_{k}$ exceeds $\binom{N-1}{k-1}\cdot (k/N)^{\alpha}$ is at most $\exp(-1000N(N/k)^{\alpha})$. Hence, with probability at least $1 - \exp(-900N(N/k)^{\alpha})$, simultaneously for all $\sigma$ and $I$, the number of bad $A_{k}$ is at most $\binom{N-1}{k-1}\cdot (k/N)^{\alpha}$.
\end{claim}
\begin{proof}
In proving this claim, we may assume without loss of generality that $k-1$ divides $N-1$. Then, to any permutation $\pi$ of $[N-1]$, we naturally associate $(N-1)/(k-1)$ disjoint sets of size $k-1$. Together with $I$, these give $(N-1)/(k-1)$ sets of size $k$: $T_1,\dots, T_{(N-1)/(k-1)}$. By symmetry and double counting, it suffices to show that for any fixed permutation, with probability at least $1-\exp(-1000N\log(N/k)\log p)$, the fraction of $T_1,\dots, T_{(N-1)/(k-1)}$ which are bad is at most $(k/N)^{\alpha}$. By \cite[Theorem~4.4.5]{vershynin2018high} and a union bound, the probability that a fixed $T_i$ is bad is at most $\exp(-Ck(N/k)^{2\alpha})$. Since the Gaussians appearing in the matrices corresponding to distinct $T_i$ are different (and in particular, independent), it follows that the probability that the number of bad $T_i$ exceeds $(N/k)^{1-\alpha}$ is at most $\exp(-CN(N/k)^\alpha)$. The final assertion follows by a union bound.
\end{proof}
The key claim in the proof of \cref{lemma:p spin main} is the following. 
\begin{claim}
\label{claim:E-exp}
With probability at least $1 - \exp(-100N)$ over the choice of the Hamiltonian $H = H(g)$,
\[\sup_{I,\sigma}\mathbb{E}_{\pi}\left[\underbrace{\exp\left(\frac{\sum_{k = k_0}^{N}\sup_{A',B'}\|\Delta_{0}^{[A',B']}\|_{\operatorname{op}}}{k}\right)}_{:=Z}\right] \lesssim 1.\]
\end{claim}

\begin{proof}
We consider a `good' realisation of the Hamiltonian i.e.~a realisation satisfying (i) the conclusion of \cref{claim:s'=0} and (ii) $\sup_{I,\sigma, A_{k}, A', B'}\|\sum_{p\geq 2}\Delta_{0,p}^{[A',B']}\| \leq c$, for a sufficiently small constant $c$ to be chosen later. By \cref{claim:s'=0}  and \cite[Lemma~6.1]{pspin}, this indeed holds with probability at least $1 - \exp(-100N)$, provided that $\beta_0:=\sum_{p\ge 2}\sqrt{p^3 \log p}\beta_p$ is sufficiently small depending on $c$. We show that for a good realisation, the conclusion of \cref{claim:E-exp} is satisfied. 

To every permutation $\pi$, we naturally associate a sequence of sets $A_{N},\dots, A_{k},\dots, A_{k_0}$. Let $N_{k}$ denote the operator norm of the submatrix corresponding to $A_{k}$. Let $\mathcal{B}_k$ denote the event that $A_k$ is bad and $\mathcal{G}_k$ denote the complementary event that $A_k$ is good. We write the random variable $Z$ as
\[Z = \exp\left(\sum_{k = k_0}^{N} \frac{N_k \cdot 1_{\mathcal{G}_k} + N_k \cdot 1_{\mathcal{B}_k}} {k}\right).\]
Deterministically,
\begin{align*}
\sum_{k= k_0}^{N} \frac{N_k \cdot 1_{\mathcal{G}_k}}{k} \lesssim \left(\sum_{p\ge 2}\beta_{p}p\sqrt{\log{p}}\right) \sum_{k = k_0}^{N}k^{-1}(k/N)^{1/2-\alpha} \lesssim 1,
\end{align*}
so it suffices to show that
\begin{align}
\label{eqn:bad-event}
    \sup_{\sigma, I}\mathbb{E}_{\pi}\left[\exp\left(\sum_{k = k_0}^{N} \frac{N_{k}\cdot 1_{\mathcal{B}_k}}{k}\right)\right] \lesssim 1.
\end{align}
This follows from H\"older's inequality. Indeed, let $q_{k} = \gamma k^{2}$, where $\gamma$ is an absolute constant ensuring that $\sum_{k = k_0}^{N}q_{k}^{-1} = 1$. By 
\cref{claim:s'=0}, $\E_{A_k} {1_{\mathcal{B}_k}} \leq (k/N)^{\alpha}$ and we can always bound $ N_k\leq c$ for a good realisation. Hence, we have
\begin{align*}
    \mathbb{E}_{A_k}\left[ \exp\left(q_k \cdot \frac{N_{k}\cdot 1_{\mathcal{B}_k}}{k}\right)\right] \leq \exp(cq_k/k)(k/N)^{\alpha},
\end{align*}
thus, using H\"older's inequality i.e. $\E{r_{k_0} \dots r_N}\leq \E{r_{k_0}^{q_{k_0}}}^{\frac{1}{q_{k_0}}} \dots \E{r_{N}^{q_{N}}}^{\frac{1}{q_N}}$ for $r_k = \exp(N_k 1_{\mathcal{B}_k}/k)$, we can bound the left hand side of \cref{eqn:bad-event} by
\begin{align*}
    \prod_{k = k_0}^{N}\exp(c/k)\prod_{k = k_0}^{N}(k/N)^{\alpha/q_k}
    \lesssim N^{c} N^{-\alpha} \lesssim 1,
\end{align*}
provided $\beta_0=\sum_{p\ge 2}\sqrt{p^3 \log p}\beta_p$ is small enough so that $c \leq \alpha ( = 1/4)$. 
\end{proof}
Combining \cref{eqn:Delta-1-bound} and \cref{claim:E-exp} finishes the proof. 
\qedhere

\end{proof}

  	\section{Applications}
\subsection{Spectral Independence from a Contractive Coupling}\label{s:simplified}
As a quick application of \cref{thm:poincare-to-si}, we provide a short and transparent alternate proof of recent results of Liu \cite{liu2021coupling} and Blanca et al. \cite{blanca2021mixing} showing, for instance, that for a measure $\mu$ on a product space, the existence of a contractive  coupling for the Glauber dynamics implies that $\mu$ is spectrally independent. 

More precisely, let $\Omega = \Omega_1\times \dots \times \Omega_n$ be a product space, let $d$ be a metric on $\Omega$, and let $\mu$ be a measure on $\Omega$. For $\kappa \in (0,1)$, we say that $\mu$ satisfies property $(\dagger)$ with parameter $\kappa$ with respect to the metric $d$ if the following holds.

$(\dagger)$ %for all (partial) pinnings $\tau$ of any subset of coordinates, the Glauber dynamics $P_{\tau}$ with respect to the conditional distribution $\mu_{\tau}$ has the property that 
For all valid configurations $(X_0, Y_0) \in \Omega_{} \times \Omega_{}$, there exists a coupling $(X_0, Y_0) \to (X_1, Y_1)$ of the Glauber dynamics $P$ satisfying
\[\mathbb{E}[d(X_1, Y_1) \mid (X_0, Y_0)] \leq \kappa d(X_0, Y_0).\]

\begin{theorem}[cf.~Theorem 1.10 in \cite{blanca2021mixing}]
If $\mu$ satisfies property $(\dagger)$ with parameter $\kappa \leq (1-\epsilon/n)$ with respect to any metric $d$, then $\mu$ is $\eta$-spectrally independent with constant $\eta = 1/\epsilon$.
\end{theorem}

\begin{proof}
Since $\kappa \leq (1-\epsilon/n)$, it follows by \cite[Theorem~13.1]{levin2017markov} that 
\[\gamma \geq 1 - \kappa \geq \frac{\epsilon}{n}.\]
% Let $d_{\min}$ and $d_{\max}$ be such that for $x\neq y \in \Omega$, $d_{\min} \leq d(x,y) \leq d_{\max}$. 
% By iterating $(\dagger)$, there is a coupling of two copies of the Markov chain, started at $X_0$ (respectively $Y_0$) such that
% \[\mathbb{E}[d(X_t, Y_t)] \leq \kappa^{t}d(X_0, Y_0) \leq e^{-t\epsilon/n}d(X_0,Y_0) \leq e^{-t\epsilon/n}d_{\max}.\]
% Therefore, by the coupling lemma \cite{levin2017markov} and Markov's inequality, for all $\delta \in (0,1)$, $\tau_{\operatorname{mix}}(\delta) \leq \frac{n}{\epsilon}\cdot \log(d_{\max}/\delta d_{\min})$. Hence, by \cref{eqn:mixing-time} applied with $\delta \to 0$, we see that the spectral gap $\gamma$ of $P$ satisfies
% \[\gamma^{-1} - 1 \leq \frac{n}{\epsilon},\]
and finally, applying \cref{thm:poincare-to-si}, we conclude that $\mu$ is $1(/\epsilon)$-spectrally independent.
\end{proof}

\begin{remark}
(1) There are two important advantages of the above theorem, compared to \cite[Theorem~1.10]{blanca2021mixing}. First, our bound on $\eta$ depends only on the contractive constant $\kappa$ and not on the underlying metric (as in \cite{blanca2021mixing}). Second, even in the case of weighted Hamming metrics, our bound on $\eta$ is a factor of $2$ better than the corresponding bound in \cite{blanca2021mixing}. This factor of $2$ is important in applications where $\epsilon \approx 1$, in which case our spectral independence bound can potentially be combined with local-to-global arguments to yield, e.g. modified log-Sobolev inequalities, with near-optimal dependence on $n$, whereas the bound of \cite{blanca2021mixing} will necessarily give a quadratically sub-optimal bound. 

(2) While the above theorem is stated only for the Glauber dynamics for simplicity, the same proof can be used to cover situations where (a) the contractive coupling is with respect to a different Markov chain and (b) the spectral gaps of this Markov chain is within a constant factor of the spectral gap of the Glauber dynamics. Indeed, this is the situation for the so-called flip dynamics for the uniform distribution on $q$-colorings of graphs of maximum degree $\Delta$ (with $q > (11/6 - \epsilon_0)\Delta$ for a small absolute constant $\epsilon_0$), and therefore, allows us to easily recover one of the main applications of \cite{liu2021coupling,blanca2021mixing}.  

(3) While the approaches in \cite{blanca2021mixing,liu2021coupling} generally lead to looser  spectral independence guarantees than ours, their methods can still be useful as a way to bound the $\infty$-norm of the influence matrix. 
\end{remark}

\subsection{Learning and Statistical Estimation}
There is a vast literature on learning exponential families and graphical models from data. In terms of learning these distributions (say in the total variation distance), there are easy to use guarantees known \emph{information-theoretically} via computationally inefficient algorithms, e.g.\ \cite{devroye2020minimax}. There are also guarantees for learning via polynomial time algorithms, e.g.\ the result of \cite{klivans2017learning}. In some settings, the computationally efficient guarantees are extremely suboptimal in terms of their sample complexity. Is this an inherent limitation?

\paragraph{An example where existing results fail: spin glass inversion.} 
Suppose we are given $m$ i.i.d.\ samples from a Sherrington-Kirkpatrick model at inverse temperature $\beta > 0$ on $n$ sites. As a reminder, this is the special case of the mixed $p$-spin model with quadratic interactions, so $\mu(x) \propto \exp(\frac{1}{2} \langle x, J x \rangle)$ where $J$ is a symmetric matrix with $J_{ij} \sim N(0,\beta^2/n)$, so $J$ is proportional to a GOE random matrix. The computationally inefficient result of \cite{devroye2020minimax} implies that obtaining total variation distance $0.01$ can be done with high probability from $m = \tilde{O}(n^2)$ samples.

Applying a state of the art algorithmic result such as \cite{wu2019sparse,klivans2017learning,vuffray2016interaction}, the best we can get is a sample complexity of $e^{O(\beta\sqrt{n})}$ for obtaining guarantees for learning this distribution in total variation distance. The reason is that the $\ell_1$ norm of a row of $J$ is approximately $\beta \sqrt{n}$, and these results depend exponentially on this quantity (or worse, on the degree) --- see e.g. Theorem 7.3 of \cite{klivans2017learning}. This is a fundamental limitation of the analyses in all of these works.

\begin{remark}
The SK model was prominently studied in a different context in statistical estimation \cite{chatterjee2007estimation}, where $J$ is known up to normalization and the goal was to estimate a single parameter (the inverse temperature $\beta$) from a single sample. Here we are trying to estimate the entire distribution which requires many more samples from the distribution. Our problem has been considered under the name of \emph{spin glass inversion} in the statistical physics literature \cite{marinari2009intrinsic,mezard2009constraint}, where heuristic message passing algorithms have been proposed.
\end{remark}
%For example...
\paragraph{A different analysis via approximate tensorization.} 
Many algorithms in the literature on learning discrete graphical models, including those referenced above, can be understood as variants of maximum likelihood (see e.g.\ \cite{van2000asymptotic}) or \emph{pseudolikelihood} estimation \cite{besag1975statistical}. Pseudolikelihood estimation minimizes the loss
\[ \hat L_p(q) :=  \sum_{j = 1}^n \frac{1}{m} \sum_{i = 1}^m\Ehat_{X}{\log q((X_i)_j\mid (X_i)_{\sim j})} \]
where $X_1,\ldots,X_m$ are i.i.d.\ samples from the ground truth distribution $p$ and $(X_i)_j$ denotes coordinate $j$ of sample $X_i$. 
It was recently observed that approximate tensorization of entropy has immediate consequences for the sample complexity of pseudolikelihood estimation \cite{koehler2022sampling}. 
\begin{theorem}[Special case of Theorem 4 of \cite{koehler2022statistical}]\label{thm:finite-sample-discrete}
Suppose that $\mathcal P$ is a class of probability distributions containing $p$ and  
$C_{AT}(\mathcal P) := \sup_{q \in \mathcal P} C_{AT}(q)$
 is the worst-case approximate tensorization constant in the class of distributions. 
Let
\[ \mathcal R_m :=  \E*_{X_1,\ldots,X_m,\epsilon_1,\ldots,\epsilon_m}{\sup_{q \in \mathcal P} \frac{1}{m} \sum_{i = 1}^m \epsilon_i \left[\sum_{j = 1}^n \log q((X_i)_j \mid (X_i)_{\sim j})\right]} \]
be the expected Rademacher complexity of the class given $m$ samples $X_1,\ldots,X_m \sim p$ i.i.d.\ and independent $\epsilon_1,\ldots,\epsilon_m \sim Uni\{\pm 1\}$ i.i.d.\ Rademacher random variables. Let $\hat p$ be the pseudolikelihood estimator from $n$ i.i.d. samples from $p$, in other words let
$\hat p = \arg\min_{q \in `\mathcal P} \hat L_p(q)$.
Then
\[ \E{\DKL{p \river \hat p}} \le 2C_{AT}(\mathcal P) \mathcal R_m. \]
\end{theorem}
Since in this work we established bounds on the approximate tensorization constant of a large class of distributions, they can be directly combined with this result to obtain new learning guarantees. The precise choice of the class $\mathcal P$ will depend on the application (larger classes $\mathcal P$ will require more sample complexity to learn). 

\paragraph{Exponential improvement in the example.} We revisit the example of learning distributions like the SK model from data. We observe that when $\beta$ is small enough so that approximate tensorization is satisfied, we get a dramatically improved guarantee for learning the distribution from samples:
\begin{theorem} \label{thm:learning ising}
Suppose $p$ is a distribution lying in $\mathcal P$, which is defined to be the class of distributions of the form
\[ p_{J,h}(x) \propto \exp\left(\frac{1}{2} \langle x, J x \rangle + \langle h, x \rangle\right) \]
under the assumption that for some $R > 0$:
\begin{enumerate}
    \item $\|J\|_{OP} \le \alpha < A$ where $A > 0$ is the constant from Theorem~\ref{thm:poincare bound}.
    \item $h_j \le R$ for every $j$, and $\|J_j\|_1 \le R$ for every row $J_j$.
\end{enumerate}
Let $\Theta$ denote the set of $(J,h)$ pairs satisfying these conditions.
 Let $\hat p$ be the pseudolikelihood estimator from $n$ i.i.d. samples from $p$, in other words let
$\hat p = \arg\min_{q \in `\mathcal P} \hat L_p(q)$. (This is a convex program which can be efficiently optimized.)
Then
\[ \E{\DKL{p \river \hat p}} = O\left(Rn^{1 + B\alpha} \sqrt{\log(n)/m}\right) \]
where $B$ is as defined in Theorem~\ref{thm:poincare bound}. 
\end{theorem}
\begin{proof}
First, observe that the conditional density is
\[ p_{J,h}(x_j \mid x_{\sim j}) = \frac{e^{\langle J_j, x \rangle x_j + h_j x_j}}{e^{\langle J_j, x \rangle x_j + h_j x_j} + e^{-\langle J_j, x \rangle x_j - h_j x_j}} = \frac{1}{1 + e^{-2\langle J_j, x \rangle x_j - 2 h_j x_j}} \]
where $J_j$ denotes row $j$ of the matrix $J$. So
\[ \log p_{J,h}(x_j \mid x_{\sim j}) = -\ell(-\langle J_j x \rangle x_j - h_j x_j) \]
where $\ell(z) = \log(1 + e^{2z})$ is the logistic loss, which is a $1$-Lipschitz function.
We can now bound the Rademacher complexity of this class:
\begin{align*}
\mathcal{R}_n 
&= \E*_{X_1,\ldots,X_n,\epsilon_1,\ldots,\epsilon_n}{\sup_{J,h \in \Theta} \frac{1}{m} \sum_{i = 1}^m \epsilon_i \left[\sum_{j = 1}^n \ell(-(X_i)_j (\langle J_j, X_i\rangle + h_j))\right]} \\
&\le \sum_{j = 1}^n \E*_{X_1,\ldots,X_n,\epsilon_1,\ldots,\epsilon_n}{\sup_{J,h \in \Theta} \frac{1}{m} \sum_{i = 1}^m \epsilon_i \ell(-(X_i)_j (\langle J_j, X_i\rangle + h_j))} \\
&\le \sum_{j = 1}^n \E*_{X_1,\ldots,X_n,\epsilon_1,\ldots,\epsilon_n}{\sup_{J,h \in \Theta} \frac{1}{m} \sum_{i = 1}^m \epsilon_i (X_i)_j (\langle J_j, X_i\rangle + h_j)} \\
&= \sum_{j = 1}^n \E*_{X_1,\ldots,X_n,\epsilon_1,\ldots,\epsilon_n}{\sup_{J,h \in \Theta} \frac{1}{m} \langle J_j, \sum_{i = 1}^m \epsilon_i X_i (X_i)_j\rangle + \frac{h_j}{m} \sum_{i = 1}^m \epsilon_i (X_i)_j} \\
&\le \sum_{j = 1}^n \E*_{X_1,\ldots,X_n,\epsilon_1,\ldots,\epsilon_n}{\frac{1}{m} R\left\|\sum_{i = 1}^m \epsilon_i X_i (X_i)_j\right\|_{\infty} + \frac{R}{m} \left|\sum_{i = 1}^m \epsilon_i (X_i)_j\right|}
\end{align*}
where in the first inequality we moved the supremum inside the sum over $j$, in the second inequality we used Talagrand's contraction principle (Exercise 6.7.7 of \cite{vershynin2018high}), and in the last inequality we used Holder's inequality. Using Hoeffding's inequality and a standard tail bound for the maximum of subgaussian random variables (Exercise 2.5.10 of \cite{vershynin2018high}), we conclude that
\[ \mathcal R_n \le n R \sqrt{\frac{2 \log n}{m}} + n R\sqrt{1/m} \le 4 Rn \sqrt{\frac{\log n}{m}}.  \]
Appealing to Theorem~\ref{thm:finite-sample-discrete} and Theorem~\ref{thm:at-general} proves the result. 
\end{proof}
Going back to the example of the SK model, we have $R = \tilde O(\beta \sqrt{n})$ so for $\beta$ sufficiently small this implies we can learn the distribution to KL divergence $o(1)$, hence also TV distance $o(1)$ by Pinsker's inequality \cite{cover2012elements}, with $m = \tilde{O}(n^{3 + 2B \beta})$ samples. This is not much worse than the sample complexity required for the inefficient method of \cite{devroye2020minimax} and exponentially improves the previously mentioned algorithmic guarantees. We remark that this analysis can be adapted to the estimators proposed in \cite{klivans2017learning,wu2019sparse,vuffray2016interaction} and other works, since their algorithms are based on similar principles to, though not identical to, the standard pseudolikelihood estimator. The key difference is our analysis of the estimator. 

This results also works for many other models, e.g.\ diluted spin glasses, with similar improvements 
(exponential improvement in the dependence on the degree $d$). Finally, we note that while the guarantee of \cref{thm:finite-sample-discrete} is in expectation, it is straightforward to obtain a strong high probability guarantee by combining the same argument with standard tools from generalization theory (see \cite{koehler2022statistical,shalev2014understanding}).

\begin{remark}
Obviously, this result (and the identity testing result in the next section) can be generalized  to the case of higher-order interactions.
In the special case of the Ising model, it is possible to tweak the above and obtain a version of the result for a larger value of $\|J\|_{OP}$ by using the result of \cite{eldan2021spectral} and appealing to comparison inequalities to bound the log-Sobolev or approximate tensorization constant; this results in a worse dependence on the dimension and on the size of the external field.
%We focused on the quadratic case largely to simplify notation. 
\end{remark}

\subsection{Identity Testing} 
\label{sec:identity}
In the identity testing problem, given an explicitly visible distribution $\mu$ and oracle access to samples from an unknown/hidden distribution $\pi,$ the goal is to
determine if these distributions are identical using as few samples from $\pi$ as possible. \cite{BCSV21} shows efficient identity testing algorithms for distributions $\mu$ satisfying approximate tensorization of entropy.
\begin{theorem}[{\cite[Theorem 7.5]{BCSV21}}]
    Consider distribution $\mu:\set{\pm 1}^{[n]} \to \R_{\geq 0}$ satisfying approximate tensorization of entropy with parameter $ C$
    and $b$-marginally bounded assumption i.e. if for every $\Lambda \subseteq [n]$, every
$x \in \set{\pm 1}^{\Lambda}$ with $\mu(X^\Lambda = x) > 0$, every $i \in [n] \setminus \Lambda$, and every $a\in \set{\pm 1}$, one has
\[\text{either } \mu(X_i=a \lvert X^{\Lambda } = x) = 0 \text{ or }\mu(X_i = a  \lvert X^{\Lambda } = x) \geq b \]
    with $b := b(n)$ satisfying $\log\log (b) = O(\log n).$
    
    Suppose that there is FPRAS to estimate the marginals of $\mu$ i.e. estimating $\mu_i(\cdot \vert x_{-i}) $ for any $ x_{-i}\in \set{\pm 1}^{[n]\setminus \set{i}}.$
    Given query access to the Subcube oracle, there exist an identity testing algorithm for KL divergence with distance parameter $\epsilon$
    with query complexity 
    \[ O\parens*{\log\parens*{\frac{1}{b}} n\log \parens*{\frac{n}{\epsilon}}}\]
\end{theorem}
%This immediately implies efficient identity testing algorithm for $\mu$ satisfying the assumption of \cref{thm:poincare bound} with bounded external field, since the marginal lower bound $b(n):=\exp(-(\theta(\sqrt{n})+ \max_i \abs{h_i}))$ can be established in a manner similar to proof of \cref{thm:}
% \begin{corollary} \label{cor:pspin identity testing}
%     Suppose $\mu:\set{\pm 1}^{[n]}\to \R_{\geq 0}$ satisfy the preconditions of \cref{thm:poincare bound} then there exists identity testing algorithm for KL divergence with distance parameter $\epsilon$ with query complexity $O( n^{3/2+B\beta} \log(n/\epsilon))$ with $B$ being a constant same as in \cref{thm:at-general}
% \end{corollary}

This immediately implies efficient identity testing algorithm for $\mu$ satisfying the preconditions \cref{thm:learning ising}, since $\mu$ is $b(n)=\exp(-\theta(\alpha\sqrt{n} +R))$-marginally bounded by the same argument as in proof of \cref{thm:learning ising}.
\begin{corollary}\label{cor:SK model identity testing}
    Let $p$ be the distribution satisfying the assumptions in \cref{thm:learning ising}. There exists identity testing algorithm for KL divergence with distance parameter $\epsilon$ with query complexity $O( n^{3/2+B\beta + R} \log(n/\epsilon))$ with $B$ being a constant, the same as in \cref{thm:at-general}.
\end{corollary}

 \paragraph{Acknowledgements.} F.K. was supported in part by NSF award CCF-1704417, NSF award IIS-1908774, and N. Anari’s Sloan Research Fellowship. V.J~was supported by NSF CAREER award DMS-2237646. 
	\PrintBibliography
	\appendix
    \section{Miscellaneous Facts} 

\paragraph{Equivalent expressions for the Dirichlet form.} 
Let
\[ \nu(\sigma) \propto \exp(H(\sigma)) \]
be a binary spin system on the hypercube $\{\pm 1\}^n$. Note that the conditional law at site $j$ is
\[ \nu(\sigma_j \mid \sigma_{\sim j}) = \frac{\exp(\sigma_j \partial_j H(\sigma))}{\exp(\sigma_j \partial_j H(\sigma)) + \exp(-\sigma_j \partial_j H(\sigma))} = \frac{1}{1 + \exp(-2\sigma_j \partial_j H(\sigma))}\]
Recall that the Dirichlet form for Glauber dynamics is
\[ \mathcal{E}(f,f) :=  \sum_{\sigma \sim \tau} \frac{\nu(\sigma)\nu(\tau)}{\nu(\sigma) + \nu(\tau)} (f(\sigma) - f(\tau))^2 = \frac{1}{2} \sum_{\sigma} \nu(\sigma) \sum_{j = 1}^n \frac{\nu(\hat \sigma_j)}{\nu(\sigma_j) + \nu(\hat \sigma_j)} (f(\sigma) - f(\hat \sigma_j))^2  \]
where $\hat \sigma_j$ denotes $\sigma$ with the spin at site $j$ flipped. Note that
\[ \frac{\nu(\hat \sigma_j)}{\nu(\sigma_j) + \nu(\hat \sigma_j)} = \frac{1}{\nu(\sigma_j)/\nu(\hat \sigma_j) + 1} = \frac{1}{1 + \exp(2 \sigma_j \partial_j H(\sigma))}. \]
This lets us establish the following fact which shows consistency with the notation in \cite{pspin}:
\begin{lemma}[Standard, see e.g.\ \cite{pspin}]
For all functions $f$, we have
\[ \mathcal{E}(f,f) = \frac{1}{4} \sum_{\sigma} \nu(\sigma) \sum_{j = 1}^n \cosh^{-2}(\partial_j H(\sigma)) (f(\sigma) - f(\hat \sigma_j))^2 \]
\end{lemma}
\begin{proof}
We have (using that $\cosh(\partial_j H) = \cosh(\sigma_j \partial_j H)$ by evenness)
\begin{align*} 
&\frac{1}{4} \sum_{\sigma} \nu(\sigma) \sum_{j = 1}^n \cosh^{-2}( \partial_j H(\sigma)) (f(\sigma) - f(\hat \sigma_j))^2 \\
&=  \sum_{\sigma} \nu(\sigma) \sum_{j = 1}^n \frac{1}{(\exp(\sigma_j \partial_j H(\sigma)) + \exp(-\sigma_j \partial_j H(\sigma)))^2} (f(\sigma) - f(\hat \sigma_j))^2 \\
&= \frac{1}{2} \sum_{\sigma} \nu(\sigma) \sum_{j = 1}^n \frac{1}{1 + (\exp(2\sigma_j \partial_j H(\sigma)) + \exp(-2\sigma_j \partial_j H(\sigma)))/2} (f(\sigma) - f(\hat \sigma_j))^2 \\
&= \frac{1}{2} \sum_{\sigma} \sum_{j = 1}^n \nu(\sigma_{\sim j})  \frac{\nu(\sigma_j \mid \sigma_{\sim j})}{1 + (\exp(2\sigma_j \partial_j H(\sigma)) + \exp(-2\sigma_j \partial_j H(\sigma)))/2} (f(\sigma) - f(\hat \sigma_j))^2 \\
&= \frac{1}{2} \sum_{\sigma} \sum_{j = 1}^n \nu(\sigma_{\sim j}) \frac{1}{2 + \exp(2\sigma_j \partial_j H(\sigma)) + \exp(-2\sigma_j \partial_j H(\sigma))} (f(\sigma) - f(\hat \sigma_j))^2
\end{align*}
where in the last step we averaged over pairs of $\sigma$  agreeing on $\sigma_{\sim j}$ and used that 
\[ \frac{\nu(\sigma_j = +1 \mid \sigma_{\sim j}) + \nu(\sigma_j = -1 \mid \sigma_{\sim j})}{2} = \frac{1}{2}. \] 
On the other hand, the Dirichlet form is
\begin{align*}
\mathcal{E}(f,f) 
&= \frac{1}{2} \sum_{\sigma} \sum_{j = 1}^n \nu(\sigma_{\sim j})  \frac{\nu(\sigma_j \mid \sigma_{\sim j})}{1 + \exp(2 \sigma_j \partial_j H(\sigma))} (f(\sigma) - f(\hat \sigma_j))^2 \\
&=  \frac{1}{2} \sum_{\sigma} \sum_{j = 1}^n \nu(\sigma_{\sim j})  \frac{1}{(1 + \exp(-2\sigma_j \partial_j H(\sigma)))(1 + \exp(2 \sigma_j \partial_j H(\sigma)))} (f(\sigma) - f(\hat \sigma_j))^2 \\
&= \frac{1}{2} \sum_{\sigma} \sum_{j = 1}^n  \nu(\sigma_{\sim j}) \frac{1}{2 + \exp(-2\sigma_j \partial_j H(\sigma)) + \exp(2 \sigma_j \partial_j H(\sigma))} (f(\sigma) - f(\hat \sigma_j))^2 
\end{align*}
so we established the desired equality. 
\end{proof}

\section{Deferred proofs \texorpdfstring{from \cref{sec:gibbs}}{}}
\label{app:deferred-proofs}

\begin{proof}[Proof of \cref{cor:at for constantly many free spins}]
%Let $A_k \subseteq [n]$ be the set of (free) unpinned spins, and supposed the remaining spins in $A_k^c$ are pinned according to an arbitrary $\sigma \in \set{\pm 1}^{A_k^c}.$ 
Consider $A \subseteq [n]$ and a pinning $\sigma_A,$ then the pinned subsystem on $[n]\setminus A$ is of the form
\[\mu^{[A,\emptyset]}_{\sigma_A}(\sigma) \propto \exp(H^{[A,\emptyset]}_{\sigma_A}(\sigma))\]
Let $H_A := H^{[A,\emptyset]}_{\sigma_A}$ and consider the multilinear extension of $H_A$: \[ H_A(x) = \sum_{S \subset [n]\setminus A} \hat H(S) \prod_{i \in S} \sigma_i \] 
Let $\mu_0 (\sigma)\propto \exp(\sum_{i\in A} \hat H_A(\set{i}) \sigma_i )$, so that $\mu_0$ is a product distribution and hence satisfies approximate entropy tensorization with $C=1.$ Let
\[H_{A,\geq 2} (\sigma) =\sum_{S \subset [n]\setminus A ,\abs{S} \geq 2} \hat H(S) \prod_{i \in S} \sigma_i,\]
so that $\mu(\sigma)\propto \exp(H_{A,\geq 2}(\sigma))\mu_0(x). $
Then,
%For any $\sigma'\in \set{\pm 1}^{A_k},$ let $x\in \set{\pm 1}^{[n]}$ s.t. $x_{A_k} = \sigma'$ and $x_{A_k^c} = \sigma$ then  
\[\abs{H_{A,\geq 2} (\sigma)} = \abs{\sum_{S\subseteq A}  \hat H(S) \prod_{i \in S} \sigma_i}  = \abs{\sigma^\intercal \nabla^2 H_{A,\geq 2} \sigma} \leq \norm{\nabla^2 H_{A,\geq 2}}_{\operatorname{OP}} \norm{\sigma}_2^2 = (N-k) \norm{\nabla^2 H_{A,\geq 2}}_{\operatorname{OP}} \leq O(\beta)  \]
by \cref{lem:submatrix norm trivial bound}. The required assertion then follows from \cref{lem:comparision} below. 
\end{proof}

\begin{lemma}[Comparison theorem for approximate entropy tensorization] \label{lem:comparision}
Consider distributions $ \mu$ and $\mu'$ over $\Omega$ satisfying $\mu(x) \propto \mu'(x) \exp(W(x)).$  %Suppose $\abs{\log(\mu(x)/\mu_0(x))}\leq \gamma$ and 
Let $\norm{W}_{\infty} = \sup_x \abs{W(x)}.$ Then, for any function $f : \Omega \to (0,\infty),$
\[\Ent_{\mu}{f} \leq \exp(2\|W\|_{\infty})\Ent_{\mu'}{f}.\]
Consequently,
if $\mu'$ satisfies approximate entropy tensorization with constant $C,$ then $ \mu$ satisfies approximate entropy tensorization with constant $\exp(6\norm{W}_{\infty}) C .$
\end{lemma}
\begin{proof}
Let $Z_{\mu} = \int d \mu$ and $Z_{\mu'}=\int d\mu'$ be normalization constants of $\mu$ and $\mu'$ respectively. It is easy to see that \[Z_{\mu'}=\int \exp(-W(x))d\mu(x) \leq \exp(\norm{W}_{\infty}) Z_{\mu}.\]
Thus $\exp(-\norm{W}_{\infty}) \mu'(x) \leq \mu(x) \leq \exp(\norm{W}_{\infty}) \mu'(x).$

By the Donsker-Varadhan theorem, 
\begin{align*}
\Ent_{\mu}{f} &= \inf_{t >0}  \int (f \log f - f \log t -f + t) d\mu
\end{align*}
where$ f\log f - f\log t -f +t = f(-\log (t/f) + (t/f-1)) \geq 0$ sine $ \log x \leq x-1 $ for $ x \in (0,\infty).$
Thus
\begin{align*}
\Ent_{\mu}{f} &= \inf_{t >0} \int (f \log f - f \log t -f + t) d\mu(x)\\
&=\inf_{t >0} \int  (f(x) \log f(x) - f(x) \log t -f(x) + t) \exp(W(x)) Z_{\mu}^{-1} Z_{\mu'} d \mu'(x)\\
&\leq \inf_{t >0} \int  (f(x) \log f(x) - f(x) \log t -f(x) + t) \exp(2\norm{W}_{\infty}) d\mu'(x)\\
&= \exp(2\norm{W}_{\infty}) \inf_{t >0} \int  (f(x) \log f(x) - f(x) \log t -f(x) + t)  d\mu'(x)\\
&= \exp(2\norm{W}_{\infty}) \Ent_{\mu'}{f}
\end{align*}
Next, since $\mu'$ satisfies approximate entropy tensorization with constant $C$
\begin{align*}
  \Ent_{\mu}{f}&\leq   \exp(2\norm{W}_{\infty}) \Ent_{\mu'}{f}\leq \exp(2\norm{W}_{\infty}) C\sum_i \E_{x_{-i} \sim \mu'}{ \Ent_{\mu'_{\vert x_{-i}}}{f_{\vert x_{-i}}} }\\
  &\leq \exp(4\norm{W}_{\infty})C\sum_i \E_{x_{-i} \sim \mu'}{ \Ent_{\mu_{\vert x_{-i}}}{f_{\vert x_{-i}}}
  }\\
  &\leq \exp(6\norm{W}_{\infty})C\sum_i \E_{x_{-i} \sim \mu}{ \Ent_{\mu_{\vert x_{-i}}}{f_{\vert x_{-i}}}
  }
\end{align*}
where in the penultimate inequality we use the first statement and the fact that \[\exp(-2\norm{W}_{\infty})\leq \mu'_{\vert x_{-i}}(x_i)/\mu_{\vert x_{-i}}(x_i) =  \frac{\mu'(x)}{\mu(x)}\cdot \frac{\mu(x_{-i})}{\mu'(x_{-i})} = \frac{\mu'(x)}{\mu(x)}  \cdot\frac{\int_{y: y_{-i} = x_{-i}}\mu(y) }{\int_{y: y_{-i} = x_{-i}}\mu'(y)} \leq \exp(2\norm{W}_{\infty}),\]
and in the last inequality, we just use $\mu'(x_{-i}) \leq \exp(2\norm{W}_{\infty}) \mu(x).$
\end{proof}

%	\PomegranateCheatSheet
\end{document}